\documentclass[11pt, twoside, leqno]{article}

\usepackage{amsfonts}

\usepackage{amssymb}
\usepackage{amsmath}
\usepackage{amsthm}
\usepackage{xcolor}
\usepackage{mathrsfs}

\allowdisplaybreaks

\def\rr{{\mathbb R}}

\def\zz{{\mathbb Z}}
\def\cc{{\mathbb C}}
\def\nn{{\mathbb N}}

\def\pp{{\mathbb P}}

\def\ca{{\mathcal A}}
\def\cb{{\mathcal B}}
\def\ccc{{\mathcal C}}
\def\cd{{\mathcal D}}
\def\ce{{\mathcal E}}
\def\cf{{\mathcal F}}

\def\cl{{\mathcal L}}

\def\cs{{\mathcal S}}

\def\cx{{\mathcal X}}

\def\sca{{\mathscr A}}
\def\scb{{\mathscr B}}

\def\scg{{\mathscr G}}

\def\sci{{\mathscr I}}

\def\sct{{\mathscr T}}
\def\scx{{\mathscr X}}
\def\scy{{\mathscr Y}}

\def\fz{\infty}
\def\az{\alpha}
\def\bz{\beta}
\def\dz{\delta}

\def\ez{\epsilon}
\def\gz{{\gamma}}

\def\lz{\lambda}

\def\oz{\omega}
\def\boz{\Omega}
\def\tz{\theta}
\def\sz{\sigma}
\def\bsz{\Sigma}
\def\vz{\varphi}

\def\lf{\left}
\def\r{\right}

\def\hs{\hspace{0.25cm}}
\def\ls{\lesssim}
\def\gls{\gtrsim}
\def\gs{\gtrsim}

\def\ov{\overline}
\def\noz{\nonumber}
\def\wz{\widetilde}

\def\st{\subset}

\def\bh{\backslash}

\def\supp{\mathop\mathrm{\,supp\,}}

\def\loc{{\mathop\mathrm{loc\,}}}
\def\diam{\mathop\mathrm{\,diam\,}}
\def\at{{\mathop\mathrm{at}}}

\def\lon{L^1(\mathcal{X})}
\def\ltw{L^2(\mathcal{X})}

\def\lp{L^p(\mathcal{X})}
\def\lq{L^q(\mathcal{X})}
\def\li{L^{\infty}(\mathcal{X})}

\def\bmo{\mathop\mathrm{\,{\rm BMO}(\mathcal{X})}}

\def\hona{H^1_\at (\mathcal{X})}

\def\lip{{\mathop\mathrm{\,Lip}}}

\newtheorem{theorem}{Theorem}[section]
\newtheorem{lemma}[theorem]{Lemma}
\newtheorem{corollary}[theorem]{Corollary}

\theoremstyle{definition}
\newtheorem{remark}[theorem]{Remark}
\newtheorem{definition}[theorem]{Definition}

\numberwithin{equation}{section}

\begin{document}

\arraycolsep=1pt

\title{\Large\bf Wavelet Characterizations of the Atomic Hardy Space
$H^1$ on Spaces of Homogeneous Type \footnotetext {\hspace{-0.35cm}
2010 {\it Mathematics Subject Classification}. Primary 42B30;
Secondary 42C40, 30L99.
\endgraf {\it Key words and phrases}.
metric measure space of homogeneous type,
Hardy space, regular wavelet, spline function.
\endgraf
Dachun Yang is supported by the National
Natural Science Foundation of China
(Grant Nos.~11571039 and 11361020),
the Specialized Research Fund for the Doctoral Program of Higher Education
of China (Grant No. 20120003110003) and the Fundamental Research
Funds for Central Universities of China
(Grant Nos.~2013YB60 and 2014KJJCA10).}}
\author{Xing Fu and Dachun Yang\,\footnote{Corresponding author}}
\date{ }
\maketitle

\vspace{-0.8cm}

\begin{center}
\begin{minipage}{13cm}
{\small {\bf Abstract}\quad Let $({\mathcal X},d,\mu)$
be a metric measure space of homogeneous type in the sense of
R. R. Coifman and G. Weiss and $H^1_{\rm at}({\mathcal X})$ be the atomic
Hardy space. Via orthonormal bases of regular wavelets
and spline functions recently constructed by P. Auscher and
T. Hyt\"onen, together with obtaining some crucial lower
bounds for regular wavelets,
the authors give an unconditional basis of
$H^1_{\rm at}({\mathcal X})$ and several
equivalent characterizations of $H^1_{\rm at}({\mathcal X})$
in terms of wavelets, which are proved useful.}
\end{minipage}
\end{center}

\section{Introduction}\label{s1}

\hskip\parindent The real variable theory of Hardy spaces
$H^p(\rr^D)$ on the $D$-dimensional Euclidean space
$\rr^D$ plays essential roles
in various fields of analysis such as harmonic analysis and
partial differential equations; see, for example, \cite{sw,s70,fs,s93}.
Meyer \cite{m92} established the equivalent characterizations of $H^1(\rr^D)$
via wavelets. Liu \cite{l92} obtained several equivalent characterizations of
the weak Hardy space $H^{1,\,\fz}(\rr^D)$ via wavelets.
Wu \cite{w92} further gave a wavelet area integral characterization of
the  weighted Hardy space $H^p_{\oz}(\rr^D)$ for $p\in(0,1]$.
Later, via the vector-valued Calder\'on-Zygmund theory,
Garc\'ia-Cuerva and Martell \cite{gm} obtained a characterization
of $H^p_{\oz}(\rr^D)$ for $p\in(0,1]$ in terms of
wavelets without compact supports.

It is well known that many classical results
of harmonic analysis over Euclidean
spaces can be extended to spaces of homogeneous type
in the sense of Coifman and Weiss \cite{cw71,cw77}, or to
RD-spaces introduced by Han, M\"uller and Yang \cite{hmy08}
(see also \cite{hmy06,yz11}).

Recall that a quasi-metric space $(\cx, d)$ equipped
with a nonnegative measure
$\mu$ is called a {\it space of homogeneous type}
in the sense of Coifman and Weiss \cite{cw71,cw77}
if $(\cx, d,\mu)$ satisfies the following {\it measure doubling condition}:
there exists a positive constant $C_{(\cx)}\in[1,\fz)$ such that,
for all balls
$B(x,r):= \{y\in\cx:\,\, d(x, y)< r\}$
with $x\in\cx$ and $r\in(0, \fz)$,
\begin{equation*}
\mu(B(x, 2r))\le C_{(\cx)} \mu(B(x,r)),
\end{equation*}
which further implies that there exists a
positive constant $\wz C_{(\cx)}$ such that,
for all $\lz\in[1,\fz)$,
\begin{equation}\label{a.b}
\mu(B(x, \lz r))\le \wz C_{(\cx)}\lz^{n} \mu(B(x,r)),
\end{equation}
where $n:=\log_2 C_{(\cx)}$. Let
\begin{equation}\label{n0}
n_0:=\inf\{n\in(0,\fz):\ n\ {\rm satisfies}\ (\ref{a.b})\}.
\end{equation}
It is obvious that $n_0$ measures the
dimension of $\cx$ in some sense and
$ n_0\le n$. Observe that \eqref{a.b} with $n$ replaced by $n_0$ may
not hold true.

A space of homogeneous type, $(\cx, d,\mu)$,
is called a \emph{metric measure space of homogeneous type}
in the sense of Coifman and Weiss
if $d$ is a metric.

Recall that an RD-\emph{space} $(\cx,d,\mu)$ is
defined to be a space of homogeneous type
satisfying the following additional \emph{reverse
doubling condition} (see \cite{hmy08}): there exist positive constants
$a_0,\ {\widehat C}_{(\cx)}\in(1,\fz)$ such that, for all balls $B(x,r)$
with $x\in\cx$ and $r\in(0, \diam(\cx)/a_0)$,
$$\mu(B(x, a_0r))\ge {\widehat C}_{(\cx)} \mu(B(x,r))$$
(see \cite{yz11} for more equivalent characterizations of RD-spaces).
Here and hereafter,
$$\diam (\cx):=\sup\{d(x,y):\ x,\,y\in\cx\}.$$

Let $(\cx,d,\mu)$ be a space of homogeneous type.
In \cite{cw77}, Coifman and Weiss introduced the
atomic Hardy space $H^{p,\,q}_\at (\cx,d,\mu)$ for all
$p\in(0,1]$ and $q\in[1,\fz]\cap(p,\fz]$ and showed
that $H^{p,\,q}_\at (\cx,d,\mu)$ is independent
of the choice of $q$, which is hereafter simply denoted by
$H^p_\at(\cx,d,\mu)$, and that its dual space is the
Lipschitz space $\lip_{1/p-1}(\cx,d,\mu)$ when $p\in(0,1)$,
or the space $\mathop\mathrm{BMO}(\cx,d,\mu)$
of functions with bounded mean oscillations when $p=1$.

Recall that Coifman and Weiss \cite{cw77}
introduced the following \emph{measure distance $\rho$}
which is defined by setting, for all $x,\,y\in\cx$,
\begin{equation}\label{a.x}
\rho(x,y):=\inf\lf\{\mu\lf(B_d\r):\
B_d\ \mathrm{is\ a\ ball\ containing}\ x\ {\rm and}\ y\r\},
\end{equation}
where the infimum is taken over all balls in $(\cx,d,\mu)$ containing
$x$ and $y$; see also \cite{ms1}.
It is well known that, although all balls defined by $d$ satisfy the axioms
of the complete system of neighborhoods in $\cx$ [and hence induce a (separated)
topology in $\cx$], the balls $B_d$ are not necessarily
open with respect to the topology
induced by the quasi-metric $d$. However, by \cite[Theorem 2]{ms1},
we see that there exists a quasi-metric $\wz{d}$ such that $\wz{d}$
is \emph{equivalent} to $d$, namely, there exists a positive
constant $C$ such that, for all $x,\,y\in\cx$,
$$
C^{-1}d(x,y)\le\wz{d}(x,y)\le Cd(x,y),
$$
and the balls in $(\cx,\wz{d},\mu)$ are open.

Recall also that a quasi-metric measure space $(\cx,\rho,\mu)$
is said to be \emph{normal} in \cite{ms1} if
there exists a fixed positive constant $C_{(\rho)}$
such that, for all $x\in\cx$ and $r\in(0,\fz)$,
$$
C_{(\rho)}^{-1}r\le\mu\lf(B_{\rho}(x,r)\r)\le C_{(\rho)}r.
$$

Assuming that all balls in $(\cx,d,\mu)$ are open,
Coifman and Weiss \cite[p.\,594]{cw77} claimed
that the topologies of $\cx$ induced, respectively, by $d$ and $\rho$ coincide and
$(\cx,\rho,\mu)$ is a normal space,
which were rigorously proved by
Mac\'ias and Segovia in \cite[Theorem 3]{ms1},
and also that the atomic Hardy space $H^p_\at(\cx,d,\mu)$
associated with $d$ and
the atomic Hardy space $H^p_\at(\cx,\rho,\mu)$
associated with $\rho$ coincide with equivalent quasi-norms
for all $p\in(0,1]$.
Mac\'ias and Segovia \cite[Theorem 2]{ms1} further showed that
there exists a normal quasi-metric
$\wz{\rho}$ such that $\wz{\rho}$ is
equivalent to $\rho$
and $\wz{\rho}$ is \emph{$\tz$-H\"older continuous}
with $\tz\in(0,1)$, namely,
there exists a positive constant $C$
such that, for all $x,\,\wz{x},\,y\in\cx$,
$$
\lf|\wz{\rho}(x,y)-\wz{\rho}(\wz{x},y)\r|\le C
\lf[\wz{\rho}(x,\wz{x})\r]^{\tz}
\lf[\wz{\rho}(x,y)+\wz{\rho}(\wz{x},y)\r]^{1-\tz}.
$$
Via establishing certain geometric measure relations
between $(\cx,d,\mu)$ and $(\cx,\rho,\mu)$,
Hu, Yang and Zhou \cite[Theorem 2.1]{hyz} rigorously verified
the claim of Coifman and Weiss \cite[p.\,594]{cw77}
on the coincidence of both atomic Hardy spaces $H^p_\at(\cx,d,\mu)$
and $H^p_\at(\cx,\rho,\mu)$, which was also used by Mac\'ias and
Segovia \cite[pp.\,271-272]{ms2}.

When $(\cx,\rho,\mu)$ is a normal quasi-metric measure space,
Coifman and Weiss \cite{cw77} further
established the molecular characterization
for $H^1_\at(\cx,\rho,\mu)$.
When $(\cx,\wz\rho,\mu)$ is a normal quasi-metric
measure space and $\wz{\rho}$ is $\tz$-H\"older
continuous, Mac\'ias and Segovia \cite{ms2}
obtained the grand maximal function characterization
for $H^p_\at(\cx,\wz\rho,\mu)$ with $p\in (\frac 1{1+\tz}, 1]$
via distributions acting on certain spaces of Lipschitz
functions; Han \cite{h94} obtained their
Lusin-area function characterization;
Duong and Yan \cite{dy03} then characterized these
atomic Hardy spaces via Lusin-area functions
associated with some Poisson semigroups;
Li \cite{l98} also obtained a characterization of
$H^p_\at(\cx,\wz\rho,\mu)$ in terms of the grand maximal function
defined via test functions introduced in \cite{hs94}.

Over RD-spaces $(\cx,d,\mu)$ with $d$ being a metric,
for $p\in (\frac {n_0}{n_0+1},1]$
with $n_0$ as in \eqref{n0}, Han, M\"uller and Yang \cite{hmy06}
developed a Littlewood-Paley theory
for atomic Hardy spaces $H^p_\at(\cx,d,\mu)$;
Grafakos, Liu and Yang \cite{gly1} established their characterizations
via various maximal functions.
Moreover, it was shown in \cite{hmy08} that these Hardy spaces
coincide with Triebel-Lizorkin spaces on $(\cx,d,\mu)$.
Some basic tools, including spaces of test functions,
approximations
of the identity and various Calder\'on reproducing formulas on RD-spaces,
were well developed in \cite{hmy06,hmy08},
in order to develop a real-variable theory of Hardy spaces or,
more generally, Besov spaces and Triebel-Lizorkin
spaces on RD-spaces.
From then on, these basic tools play important roles in
harmonic analysis on RD-spaces (see, for example,
\cite{glmy, gly, hmy06, hmy08, kyz10, kyz11, yz08, yz11}).

Recently, Auscher and Hyt\"onen \cite{ah13}
built an orthonormal basis
of H\"older continuous wavelets with exponential decay
via developing randomized dyadic structures and properties of
spline functions over general spaces of homogeneous type.
Motivated by \cite{ah13}, in this article,
we obtain an unconditional basis of
$H^1_\at ({\mathcal X},d,\mu)$ and several
equivalent characterizations of
$H^1_\at ({\mathcal X},d,\mu)$ in terms of wavelets.

We point out that the main result (Theorem \ref{tc.d} below) of this article 
was applied in \cite{fyl} to confirm the conjecture suggested by
A. Bonami and F. Bernicot affirmatively 
(This conjecture was presented by L. D. Ky in \cite{ky}).
More applications are also expectable. 

Throughout this article, for the presentation simplicity,
we \emph{always assume} that $(\cx,d,\mu)$
is a metric measure space of homogeneous type, $\diam (\cx)=\fz$
and $(\cx,d,\mu)$ is non-atomic, namely, $\mu(\{x\})=0$ for any
$x\in\cx$. It is known that, if $\diam (\cx)=\fz,$
then $\mu(\cx)=\fz$ (see, for example, \cite[Lemma 8.1]{ah13}).
Also, from now on, for the notational simplicity,
on function spaces over $(\cx,d,\mu)$ such as
$H^1_\at ({\mathcal X},d,\mu)$, we will simply
write it as $H^1_\at (\cx)$ by omitting $d$ and $\mu$.

The organization of this paper is as follows.

In Section \ref{s2}, we first recall some preliminary notions
on wavelets and discover some crucial lower bounds for regular wavelets
via the continuous functional calculus
(see Theorem \ref{tb.m} below).

In Section \ref{s3}, we give an unconditional basis of $\hona$.
To this end, we first establish two useful lemmas which are generalizations
of \cite[Proposition 8.8 and Corollary 7.10]{w97}, respectively.
Via these, we show that the orthonormal basis of regular wavelets is just
an unconditional basis of $\hona$, where the molecular characterization
of $\hona$ from \cite{hyz}
and the boundedness of Calder\'on-Zygmund operators from \cite{yz08}
play important roles.

Section \ref{s4} is devoted to the equivalent wavelet
characterizations of $\hona$.
Via the unconditional basis of $\hona$ in Section \ref{s3},
combined with the aforementioned obtained lower bounds for regular wavelets,
the Lebesgue differential theorem associated to
the dyadic cubes (see Lemma \ref{lc.h} below),
and the technical Lemma \ref{lc.i},
we then finish the proof of Theorem \ref{tc.d},
the equivalent characterizations of
$\hona$ via wavelets.

Finally, we make some conventions on notation.
Throughout the whole paper, $C$ stands for a {\it positive constant} which
is independent of the main parameters, but it may vary from line to
line. Moreover, we use $C_{(\rho,\,\az,\,\ldots)}$
to denote a positive constant depending
on the parameters $\rho,\,\az,\,\ldots$.
Usually, for a ball $B$, we use $c_{B}$ and $r_{B}$, respectively, to denote
its center and radius. Moreover, for any
$x,\,y\in\cx$, $r,\,\rho\in(0,\fz)$ and ball $B:=B(x,r)$,
$$
\rho B:=B(x,\rho r), \quad V(x,r):=\mu(B(x,r))=:V_r(x),
\quad V(x,y):=\mu(B(x,d(x,y))).
$$
If, for two real functions $f$ and $g$, $f\le Cg$, we then write $f\ls g$;
if $f\ls g\ls f$, we then write $f\sim g$.
For any subset $E$ of $\cx$, we use
$\chi_E$ to denote its {\it characteristic function}.
Furthermore, $\langle\cdot,\cdot\rangle$ and $(\cdot,\cdot)$
represent the duality relation and the $\ltw$ inner product,
respectively.

\section{Preliminaries on Wavelets over $(\cx,d,\mu)$}\label{s2}

\hskip\parindent In this section, we first recall
some preliminary notions and then obtain some crucial
lower bounds for regular wavelets from \cite{ah13}.

The following notion of the geometrically doubling
is well known in analysis on metric spaces, for example,
it can be found in Coifman and Weiss \cite[pp.\,66-67]{cw71}.

\begin{definition}\label{db.b}
A metric space $(\cx,d)$ is said to be \emph{geometrically doubling} if there
exists some $N_0\in \nn$ such that, for any ball
$B(x,r)\st \cx$ with $x\in\cx$ and $r\in(0,\fz)$,
there exists a finite ball covering $\{B(x_i,r/2)\}_i$ of
$B(x,r)$ such that the cardinality of this covering is at most $N_0$,
where, for all $i$, $x_i\in\cx$.
\end{definition}

\begin{remark}\label{rb.l}
Let $(\cx,d)$ be a geometrically doubling metric space.
In \cite{h10}, Hyt\"onen showed that
the following statements are mutually equivalent:
\vspace{-0.25cm}
\begin{itemize}
  \item[\rm(i)] $(\cx,d)$ is geometrically doubling.
\vspace{-0.25cm}
  \item[\rm(ii)] For any $\ez\in (0,1)$ and any ball $B(x,r)\st \cx$
with $x\in\cx$ and $r\in(0,\fz)$,
there exists a finite ball covering $\{B(x_i,\ez r)\}_i$,
with $x_i\in\cx$ for all $i$, of
$B(x,r)$ such that the cardinality of this covering
is at most $N_0\ez^{-G_0}$,
here and hereafter, $N_0$ is as in Definition \ref{db.b} and
$G_0:=\log_2N_0$.
\vspace{-0.25cm}
  \item[\rm(iii)] For every $\ez\in (0,1)$, any ball $B(x,r)\st \cx$
with $x\in\cx$ and $r\in(0,\fz)$ contains
at most $N_0\ez^{-G_0}$ centers of disjoint balls $\{B(x_i,\ez r)\}_i$
with $x_i\in\cx$ for all $i$.
\vspace{-0.25cm}
  \item[\rm(iv)] There exists $M\in \nn$ such that any ball $B(x,r)\st \cx$
with $x\in\cx$ and $r\in(0,\fz)$ contains at most $M$ centers
$\{x_i\}_i\st\cx$ of
  disjoint balls $\{B(x_i, r/4)\}_{i=1}^M$.
  \end{itemize}
\end{remark}

Recall that metric measure spaces of
homogeneous type are geometrically doubling,
which was proved
by Coifman and Weiss in \cite[pp.\,66-68]{cw71}.

Before we introduce the orthonormal basis of regular wavelets from \cite{ah13},
we first recall some notions and notation from \cite{ah13}.
For every $k\in\zz$, a set of \emph{reference dyadic points},
$\{x^k_\az\}_{\az\in\sca_k}$, here and hereafter,
\begin{equation}\label{b.v}
\sca_k\ {\rm denotes\ some\ countable\ index\ set\ for\ each}\ k\in\zz,
\end{equation}
is chosen as follows [the Zorn lemma
(see \cite[Theroem I.2]{rs80})
is needed since we consider the maximality]. For
$k=0$, let
$\scx^0:=\{x^0_{\az}\}_{\az\in\sca_0}$ be
a maximal collection of $1$-separated points. Inductively,
for any $k\in\nn$, let
\begin{equation}\label{2.1x}
\scx^k:=\{x^k_\az\}_{\az\in\sca_k}\supset\scx^{k-1}\quad
{\rm and}\quad \scx^{-k}:=\{x^{-k}_\az\}_{\az\in\sca_k}\st\scx^{-(k-1)}
\end{equation}
be maximal $\dz^k$-separated and $\dz^{-k}$-separated collections
in $\cx$ and in $\scx^{-(k-1)}$, respectively.
Fix $\dz$ a small positive parameter, for example, it
suffices to take $\dz\le\frac{1}{1000}$.
From \cite[Lemma 2.1]{ah13}, it follows that
$$
d\lf(x^k_\az,x^k_\bz\r)\ge\dz^k\ {\rm for\ all\ }\az,\,\bz\in\sca_k\
{\rm and}\ \az\neq\bz, \quad
d\lf(x,\scx\r):=\inf_{\az\in\sca_k}d\lf(x,x^k_{\az}\r)<2\dz^k.
$$

It is obvious that the dyadic reference points
$\{x^k_\az\}_{k\in\zz,\,\az\in\sca_k}$ satisfy \cite[(2.3) and (2.4)]{hk12}
with $A_0=1$, $c_0=1$ and $C_0=2$, which further induces a
dyadic system of dyadic cubes over geometrically doubling
metric spaces as in \cite[Theorem 2.2]{hk12}. We re-state it
in the following theorem, which is applied to the construction
of the orthonormal basis of regular wavelets
as in \cite{ah13}.

\begin{theorem}\label{tb.c}
Let $(\cx,d)$ be a geometrically doubling metric space.
Then there exist families of sets,
$\mathring{Q}^k_{\az}\st Q^k_{\az}\st\ov{Q}^k_{\az}$
(called, respectively, \emph{open, half-open and
closed dyadic cubes}) such that:

\begin{itemize}
\item[\rm(i)] $\mathring{Q}^k_{\az}$ and $\ov{Q}^k_{\az}$
denote, respectively, the interior and the closure of $Q^k_{\az}$;

\item[\rm(ii)] if $\ell\in\zz\cap[k,\fz)$ and $\az,\,\bz\in\sca_k$,
 then either $Q^\ell_{\bz}\st Q^k_{\az}$ or $Q^k_{\az}\cap Q^\ell_{\bz}=\emptyset$;

\item[\rm(iii)] for any $k\in\zz$,
\begin{equation*}\label{b.h}
\cx=\bigcup_{\az\in\sca_k}Q^k_{\az}\quad (disjoint\ union);
\end{equation*}

\item[\rm(iv)] for any $k\in\zz$ and $\az\in\sca_k$ with $\sca_k$ as in \eqref{b.v},
\begin{equation*}\label{b.i}
B\lf(x^k_{\az},\frac13\dz^k\r)\st Q^k_{\az}
\st B\lf(x^k_{\az},4\dz^k\r)=:B\lf(Q^k_{\az}\r);
\end{equation*}

\item[\rm(v)] if $k\in\zz,\ \ell\in\zz\cap[k,\fz),\ \az,\,\bz\in\sca_k$ and
$Q^\ell_{\bz}\st Q^k_{\az}$, then $B(Q^\ell_{\bz})\st B(Q^k_{\az}).$

\end{itemize}
The open and closed cubes $\mathring{Q}^k_{\az}$ and $\ov{Q}^k_{\az}$,
with $(k,\az)\in\sca$, here and hereafter,
\begin{equation}\label{b.s}
\sca:=\{(k,\az):\ k\in\zz,\ \az\in\sca_k\},
\end{equation}
depend only on the points $x^\ell_\bz$ for
$\bz\in\sca_{\ell}$ and $\ell\in\zz\cap[k,\fz)$.
The half-open cubes $Q^k_{\az}$,
with $(k,\az)\in\sca$, depend on $x^\ell_\bz$ for $\bz\in\sca_{\ell}$
and $\ell\in\zz\cap[\min\{k,k_0\},\fz)$,
where $k_0\in\zz$ is a preassigned number
entering the construction.
\end{theorem}

\begin{remark}\label{rb.d}
(i) In what follows, let $\le$ be the
\emph{partial order} for dyadic points as in \cite[Lemma 2.10]{hk12}.
It was shown in \cite[Lemma 2.10]{hk12} with $C_0=2$
that, if
$k\in\zz$, $\az\in\sca_k$ with $\sca_k$ as in \eqref{b.v},
$\bz\in\sca_{k+1}$ and
$(k+1,\bz)\le(k,\az)$, then
$d(x_\bz^{k+1},x_\az^k)<2\dz^k.$

(ii) For any $(k,\az)\in\sca$, let
\begin{equation}\label{b.k}
L(k,\az):=\lf\{\bz\in\sca_{k+1}:\ (k+1,\bz)\le(k,\az)\r\}.
\end{equation}
By the proof of \cite[Theorem 2.2]{hk12}
and the geometrically doubling property, we have the following
conclusions: $1\le\#L(k,\az)\le {\wz N}_0$ and
\begin{equation}\label{2.10x}
Q^k_{\az}=\bigcup_{\bz\in L(k,\,\az)}Q^{k+1}_{\bz},
\end{equation}
where ${\wz N}_0\in\nn$ is independent of $k$ and $\az$.
Here and hereafter, for any finite set $\mathfrak{C}$,
$\#\mathfrak{C}$ denotes its \emph{cardinality}.
\end{remark}

The following useful estimate about the $1$-separated set
is from \cite[Lemma 6.4]{ah13}.

\begin{lemma}\label{lb.x}
Let $\Xi$ be a $1$-separated set in a geometrically doubling
metric space $(\cx,d)$ with positive constant $N_0$.
Then, for all $\ez\in(0,\fz)$, there exists a positive constant
$C_{(\ez,\,N_0)}$, depending on $\ez$ and $N_0$, such that
$$
\sup_{a\in\cx}e^{\ez d(a,\,\Xi)/2}\sum_{b\in\Xi}
e^{-\ez d(a,\,b)}\le C_{(\ez,\,N_0)},
$$
here and hereafter, for any set $\Xi\st\cx$ and $x\in\cx$,
$d(x,\Xi):=\inf_{a\in\Xi}d(x,a)$.
\end{lemma}

Now we recall more notions and
notation from \cite{ah13}. Let
$(\Omega,\mathscr{F},\mathbb{P}_{\omega})$ be
the \emph{natural probability measure space}
with the same notation as in \cite{ah13}, where $\mathscr{F}$ is
defined as the smallest $\sigma$-algebra containing the set
$$\lf\{\prod_{k\in\zz}A_k:\ A_k\st\Omega_k:={\{0,1,\ldots,L\}
\times\{1,\ldots,M\}}\ {\rm and\ only\ finite\ many}\
A_k\neq \Omega_k\r\},
$$
where $L$ and $M$ are defined as in \cite{ah13}.
For every $(k,\az)\in\sca$ with $\sca$ as in \eqref{b.s},
the spline function is defined by setting
$$s^k_\az(x):=\mathbb{P}_\oz\lf(
\lf\{\oz\in\boz:\ x\in\overline{Q}^k_\az(\oz)\r\}\r),
\quad x\in\cx.$$
Then the splines have the following properties:
\begin{enumerate}
\item[(i)] for all $(k,\az)\in\sca$ and $x\in\cx$,
$
\chi_{B(x^k_\az,\,\frac18\dz^k)}(x)
\le s^k_\az(x)\le\chi_{B(x^k_\az,\,8\dz^k)}(x);
$

\item[(ii)] for all $k\in\zz$, $\az,\,\bz\in\sca_k$,
with $\sca_k$ as in \eqref{b.v}, and $x\in\cx$,
\begin{equation}\label{b.y}
s^k_\az(x^k_\bz)=\dz_{\az\bz},\quad
\sum_{\az\in\sca_k}s^k_\az(x)=1\quad {\rm and}\quad
s^k_\az(x)=\sum_{\bz\in\sct_{k+1}}p^k_{\az\bz}s^{k+1}_\bz(x),
\end{equation}
where,  for each $k\in\zz$, $\sct_{k+1}\st\sca_{k+1}$
denotes some countable index set
$$
\dz_{\az\bz}:=\left\{
\begin{array}{cc}
1,\ \ &\ \ {\rm if}\ \ \az=\bz,\\
0,\ \ &\ \ {\rm if}\ \ \az\neq\bz,
\end{array}
\right.
$$
and $\{p^k_{\az\bz}\}_{\bz\in\sct_{k+1}}$ is
a finite nonzero set of nonnegative numbers with
$p^k_{\az\bz}\le1$ for all $\bz\in\sct_{k+1}$;

\item[(iii)] there exist positive constants $\eta\in(0,1]$ and $C$,
independent of $k$ and $\az$, such that,
for all $(k,\az)\in\sca$ and $x,\,y\in\cx$,
\begin{equation}\label{b.1}
\lf|s^k_\az(x)-s^k_\az(y)\r|\le C\lf[\frac{d(x,y)}{\dz^k}\r]^\eta.
\end{equation}
\end{enumerate}

By \cite[Theorem 5.1]{ah13}, we know that there exists a linear,
bounded uniformly on $k\in\zz$, and injective map $U_k:\ell^2(\sca_k)
\to\ltw$ with closed range, defined by
$$
U_k\lz:=\sum_{\az\in\sca_k}\frac{\lz_{\az}}{\sqrt{\mu^k_\az}}s^k_\az,
\quad \lz:=\lf\{\lz^k_\az\r\}_{\az\in\sca_k}\in\ell^2(\sca_k),
$$
here and hereafter, $\mu^k_\az:=\mu(B(x^k_\az,\dz^k))=:V(x^k_\az,\dz^k)$
for all $(k,\az)\in\sca$, $\ell^2(\sca_k)$ denotes the space of all
sequences $\lz:=\{\lz^k_\az\}_{\az\in\sca_k}\subset\cc$ such that
$$
\|\lz\|_{\ell^2(\sca_k)}:=\lf\{\sum_{\az\in\sca_k}
\lf|\lz^k_\az\r|^2\r\}^{1/2}<\fz.
$$

Observe that, if $k\in\zz$, $\lz,\,\wz{\lz}\in\ell^2(\sca_k)$,
$f=U_k\lz$ and $\wz{f}=U_k\wz{\lz}$, then
$$
\lf(f,\wz{f}\r)_{\ltw}=\lf(M_k\lz,\wz{\lz}\r)_{\ell^2(\sca_k)}
$$
with $M_k$ being the infinite matrix which has entries
$M_k(\az,\bz)=\frac{(s^k_\az,s^k_\bz)_{\ltw}}
{\sqrt{\mu^k_\az\mu^k_\bz}}$ for $\az,\,\bz\in\sca_k$. Let $U^*_k$
be the \emph{adjoint operator} of $U_k$ for all $k\in\zz$.
Thus, for each $k\in\zz$, $M_k=U_k^*U_k$ is bounded,
invertible, positive and self-adjoint on $\ell^2(\sca_k)$.
Let $V_k:=U_k(\ell^2(\sca_k))$ for all $k\in\zz$.
The following result from \cite{ah13} (with $\mu^k_\az$ replaced
by $\nu^k_\az:=\int_\cx s^k_\az\,d\mu$)
shows that $\{V_k\}_{k\in\zz}$ is a \emph{multiresolution analysis}
(for short, MRA) of $\ltw$.

\begin{theorem}\label{tb.n}
Suppose that $(\cx,d,\mu)$ is a metric measure space of homogeneous type.
Let $k\in\zz$ and $V_k$ be the closed linear span of
$\{s^k_\az\}_{\az\in\sca_k}$. Then
$V_k\st V_{k+1}$, $\overline{\bigcup_{k\in\zz}V_k}=\ltw$
and $\bigcap_{k\in\zz}V_k=\{0\}$.

Moreover, the functions $\{s^k_\az/\sqrt{\nu^k_\az}\}_{\az\in\sca_k}$
form a Riesz basis of $V_k$: for all sequences of complex numbers
$\{\lz_\az^k\}_{\az\in\sca_k}$,
$$
\lf\|\sum_{\az\in\sca_k}\lz^k_\az s^k_\az\r\|_{\ltw}
\sim\lf[\sum_{\az\in\sca_k}\lf|\lz_\az^k\r|^2\nu^k_\az\r]^{1/2}
$$
with equivalent positive constants independent of $k$ and
$\{\lz_\az^k\}_{\az\in\sca_k}$.
\end{theorem}

Recall that \cite[Theorem 6.1]{ah13} gives the system
$\{\wz{s}^k_\az\}_{\az\in\sca_k}$ of biorthogonal
splines in $V_k$ satisfying with
$
\lf(s^k_\az,\wz{s}^k_\bz\r)_{\ltw}=\dz_{\az\bz}
$.
Now we sketch the construction of the wavelet basis
$\{\psi_{\bz}^k\}_{k\in\zz,\,\bz\in\scg_k}$,
here and hereafter, for all $k\in\zz$,
\begin{equation}\label{b.w}
\scg_k:=\sca_{k+1}\bh\sca_k
\end{equation}
with $\sca_k$ as in \eqref{b.v}.
Let $k\in\zz$. The inverse of $U_{k+1}$,
$U^{-1}_{k+1}:\ f\mapsto\{f(x^{k+1}_\bz)\sqrt{\mu^{k+1}_\bz}\}
_{\bz\in\sca_{k+1}}$,
is an isomorphism from $V_{k+1}$ onto $\ell^2(\sca_{k+1})$. Let
$$
Y_k:=U_{k+1}\lf(\lf\{\lz:=\lf\{\lz^{k+1}_\bz\r\}_{\bz\in\sca_{k+1}}
\in\ell^2(\sca_{k+1}):\
\lz^{k+1}_\bz=0\ {\rm for\ all\ }\bz\in\sca_k\r\}\r),
$$
which is identified with $U_{k+1}(\ell^2(\scg_k))$. Obviously,
$V_{k+1}=V_k\oplus Y_k$. Let $W_k$ be the \emph{orthogonal
complement} (in $\ltw$) of $V_k$ in $V_{k+1}$ and $Q_k$ the \emph{orthogonal
projection} onto $W_k$. Then the restriction of $Q_k$ to $Y_k$
is an isomorphism from $Y_k$ onto $W_k$. Then
$\{s^{k+1}_\bz\}_{\bz\in\scg_k}$
is an unconditional basis of $Y_k$ and its image under $Q_k$
is an unconditional basis of $W_k$. Thus, for all $f\in V_{k+1}$,
$$
Q_kf=f-\sum_{\az\in\sca_k}\lf(f,\wz{s}^k_\az\r)_{\ltw}s^k_\az.
$$

Moreover, the matrix $\{\wz{M}(\az,\bz)\}_{(\az,\,\bz)\in\scg_k\times\scg_k}$,
with
\begin{equation}\label{2.17x}
\wz{M}(\az,\bz):=\frac{(Q_ks^{k+1}_\az,Q_ks^{k+1}_\bz)_{\ltw}}
{\sqrt{V(y^k_\az,\dz^k)V(y^k_\bz,\dz^k)}}
=\frac{(s^{k+1}_\az,s^{k+1}_\bz)_{\ltw}}
{\sqrt{V(y^k_\az,\dz^k)V(y^k_\bz,\dz^k)}}
\end{equation}
for all $(\az,\bz)\in\scg_k\times\scg_k$,
is bounded uniformly, invertible, positive and self-adjoint
on $\ell^2(\scg_k)$, where $y^k_\bz:=x^{k+1}_\bz$ for all
$k\in\zz$ and $\bz\in\scg_k$.
Indeed, observe that $\wz{M}=(U_k|_{\ell^2(\scg_k)})^*
(U_k|_{\ell^2(\scg_k)})$, where $U_k|_{\ell^2(\scg_k)}$
denotes the restriction of $U_k$ to $\ell^2(\scg_k)$
whose \emph{adjoint operator} is denoted by
$(U_k|_{\ell^2(\scg_k)})^*$, is
the restriction of $M_{k+1}$ to $\ell^2(\scg_k)$,
which implies the desired result.

From \cite[Theorem 12.33]{r91}, it follows that
$\wz{M}^{-1/2}$ exists and is bounded, invertible,
positive and self-adjoint on $\ell^2(\scg_k)$.
Then the wavelet functions are defined by setting,
for all $k\in\zz$, $\az\in\scg_k$ and $x\in\cx$,
\begin{equation}\label{2.17x2}
\psi^k_\az(x):=\lf(U_k|_{\ell^2(\scg_k)}\r)\wz{M}^{-1/2}
\dz^{k+1}_\az(x)=\sum_{\bz\in\scg_k}\wz{M}^{-1/2}(\az,\bz)
\frac{s^{k+1}_\bz(x)}{\sqrt{\mu^{k+1}_\bz}},
\end{equation}
where $\{\dz^{k+1}_\bz\}_{\bz\in\scg_k}$ is the
\emph{canonical orthonormal basis} of $\ell^2(\scg_k)$.

Now we are ready to introduce the following
notable orthonormal basis of regular wavelets constructed by
Auscher and Hyt\"onen (\cite[Theorem 7.1]{ah13}) with
a slight difference on the notation
$$
\lf\{\psi_{\az,\,\bz}^k\r\}_{(k,\,\az)\in\wz{\sca},\,
\bz\in\wz{L}(k,\,\az)}:=\lf\{\psi_\bz^k\r\}_{k\in\zz,\,
\bz\in\scg_k},
$$
where
\begin{equation}\label{2.19x1}
\wz{\sca}:=\{(k,\az)\in\sca:\ \#L(k,\az)>1\}
\end{equation}
and, for all $(k,\az)\in\wz{\sca}$,
\begin{equation}\label{2.19x2}
\wz{L}(k,\az):=L(k,\az)\bh\{\az\},
\end{equation}
via the fact that, for any $k\in\zz$,
$$
\sca_{k+1}\bh\sca_k=\bigcup_{\{\az\in\sca_k:\ \#L(k,\,\az)>1\}}
\wz{L}(k,\az).
$$

\begin{theorem}\label{tb.a}
Let $(\cx,d,\mu)$ be a metric measure space of homogeneous type.
Then there exists an
orthonormal basis $\{\psi_{\az,\,\bz}^k\}_{(k,\,\az)\in\wz{\sca},\,
\bz\in\wz{L}(k,\,\az)}$
of $\ltw$ and positive constants $\eta\in(0,1]$
as in \eqref{b.1}, $\nu$ and $C_{(\eta)}$,
independent of $k$, $\az$ and $\bz$, such that
\begin{equation}\label{b.a}
\lf|\psi_{\az,\,\bz}^k(x)\r|\le\frac{C_{(\eta)}}
{\sqrt{V(x_{\bz}^{k+1},\dz^k)}}
e^{-\nu\dz^{-k}d(x_{\bz}^{k+1},\,x)}\ \ {\rm for\ all}\ x\in\cx,
\end{equation}
\begin{equation}\label{b.b}
\lf|\psi_{\az,\,\bz}^k(x)-\psi_{\az,\,\bz}^k(y)\r|
\le\frac{C_{(\eta)}}{\sqrt{V(x_{\bz}^{k+1},\dz^k)}}
\lf[\frac{d(x,y)}{\dz^k}\r]^{\eta}
e^{-\nu\dz^{-k}d(x_{\bz}^{k+1},\,x)}
\end{equation}
for all $x,\,y\in\cx$ satisfying $d(x,y)\le\dz^k$, and
\begin{equation}\label{b.c}
\int_{\cx}\psi_{\az,\,\bz}^k(x)\,d\mu(x)=0.
\end{equation}
\end{theorem}

Now we give out an important property of $\psi^k_\az$
which is crucial to the succeeding context.

\begin{theorem}\label{tb.m}
Let $(\cx,d,\mu)$ be a metric measure space of homogeneous type.
Then there exist positive constants $\ez_0$ and $C$,
independent of $k$, $\az$ and $\bz$, such
that, for all $(k,\az)\in\wz{\sca}$ with $\wz{\sca}$ as in
\eqref{2.19x1},
$\bz\in\wz{L}(k,\az)$ with $\wz{L}(k,\az)$ as in \eqref{2.19x2},
and $x\in B(y^k_\bz,\ez_0\dz^k)\st Q^k_\az$,
$$
\lf|\psi^k_{\az,\,\bz}(x)\r|\ge C\frac1{\sqrt{\mu(Q^k_\az)}}.
$$
\end{theorem}

\begin{proof}
Let $(k,\az)\in\wz{\sca}$ and $\bz\in\wz{L}(k,\az)$.
We first show that
\begin{equation}\label{b.n}
\lf|\wz{M}^{1/2}(\bz,\bz)\r|\ge c_3,
\end{equation}
where $\wz{M}:=\{\wz{M}(\az,\bz)\}_{(\az,\,\bz)\in\scg_k\times\scg_k}$
is as in \eqref{2.17x} and
$c_3$ is a positive constant independent of
$\az$, $\bz$ and $k$.
To this end, we adopt an idea from the proof of
\cite[Theorem 5]{l89}; see also the proof of \cite[Lemma 6.5]{ah13}.

Indeed, denote the \emph{spectrum} and the \emph{resolvent set} of
$\wz{M}$ by $\sz(\wz{M})$ and $\rho(\wz{M})$, respectively.
The spectral radius of $\wz{M}$ is defined by setting
$r(\wz{M}):=\sup\{|\lz|:\ \lz\in\sz(\wz{M})\}$.
Then, since $\wz{M}$ is positive and self-adjoint, it follows that
$\{\lz\in(0,\fz):\ \lz>r(\wz{M})\}\st\rho(\wz{M})$
and hence
$$\sz\lf(\wz{M}\r)\st\lf\{\lz\in(0,\fz):\ \lz\le r\lf(\wz{M}\r)\r\}.$$
Furthermore, by the facts that $\ell^2(\scg_k)$,
with $\scg_k$ as in \eqref{b.w}, is a Hilbert space
and that $\wz{M}$ is self-adjoint, and \cite[Theorem VI.6]{rs80},
we see that $r(\wz{M})=\|\wz{M}\|_{\cl(\ell^2(\scg_k))}$, which,
combined with the fact that $\wz{M}$ is positive,
invertible and bounded, implies that
$\sz(\wz{M})\st[a,b]$ for some $a,b\in(0,\fz)$ satisfying
$0<a<b\le\|\wz{M}\|_{\cl(\ell^2(\scg_k))}$,
since $\sz(\wz{M})$ is closed (see \cite[Theorem VI.5]{rs80}).
Here and hereafter, for a normed linear space $E$ and a
bounded linear operator $T$ from $E$ to $E$, we use $\|T\|_{\cl(E)}$
to denote the \emph{operator norm} of $T$.

Now we claim that there exists a positive, bounded and self-adjoint
operator $A$ such that $\wz{M}=2\|\wz{M}\|_{\cl(\ell^2(\scg_k))}
(Id-A)$, where $Id$ denotes the \emph{identity operator} on $\ell^2(\scg_k)$,
and
$$
\|A\|_{\cl(\ell^2(\scg_k))}
\le1-\frac{a}{2\|\wz M\|_{\cl(\ell^2(\scg_k))}}<1.
$$
Indeed, let $g(t):=1-\frac{t}{2\|\wz M\|_{\cl(\ell^2(\scg_k))}}$
for all $t\in\sz(\wz{M})$ and $A:=g(\wz{M})$.
From \cite[Theorem VII.1(e),\,(g)]{rs80}, we deduce that
\begin{eqnarray*}
\sz(A)&&=\sz\lf(g\lf(\wz{M}\r)\r)
=\lf\{g(t):\ t\in\sz\lf(\wz{M}\r)\r\}
\st\lf\{g(t):\ t\in[a,b]\r\}\\
&&=\lf[1-\frac{b}{2\|\wz M\|_{\cl(\ell^2(\scg_k))}},
1-\frac{a}{2\|\wz M\|_{\cl(\ell^2(\scg_k))}}\r]
\end{eqnarray*}
and
\begin{eqnarray*}
\|A\|_{\cl(\ell^2(\scg_k))}=\lf\|g\lf(\wz{M}\r)\r\|_{\cl(\ell^2(\scg_k))}
=\|g\|_{L^{\fz}(\sz(\wz{M}))}\le1-\frac{a}{2\|\wz M\|_{\cl(\ell^2(\scg_k))}}
<1,
\end{eqnarray*}
which show the above claim.

Thus, $(Id-A)^{-1/2}=\sum_{n=0}^\fz p_n A^n$, where $0<p_n\ls n^{1/2}$
for all $n\in\zz_+:=\nn\cup\{0\}$.
Observe that, for all $\bz\in\scg_k$,
$$
A(\bz,\bz)=1-\frac{\wz{M}(\bz,\bz)}
{2\|\wz M\|_{\cl(\ell^2(\scg_k))}}
=1-\frac{(\wz{M}\dz_\bz^{k+1},\dz_\bz^{k+1})_{\ell^2(\scg_k)}}
{2\|\wz M\|_{\cl(\ell^2(\scg_k))}}\ge1-1/2=1/2
$$
with $\{\dz^{k+1}_\bz\}_{\bz\in\scg_k}$ being the
canonical orthonormal basis of $\ell^2(\scg_k)$.

By this, \cite[Theorem VII.1(e),\,(g)]{rs80}, and the fact
that $\wz M$ is bounded uniformly on $k$, we conclude that,
for all $(k,\,\az)\in\wz{\sca}$ and $\bz\in\wz{L}(k,\,\az)$,
\begin{eqnarray*}
\wz{M}^{-1/2}(\bz,\bz)
&&=\lf(2\lf\|\wz M\r\|_{\cl(\ell^2(\scg_k))}\r)^{-1/2}(Id-A)^{-1/2}(\bz,\bz)\\
&&=\lf(2\lf\|\wz M\r\|_{\cl(\ell^2(\scg_k))}\r)^{-1/2}
\sum_{n=0}^\fz p_n A^n(\bz,\bz)\\
&&\ge\lf(2\lf\|\wz M\r\|_{\cl(\ell^2(\scg_k))}\r)^{-1/2} p_1 A(\bz,\bz)\gs1,
\end{eqnarray*}
where we used the fact that, if the infinite matrix
$A^n$ ($n\in\zz_+$) is positive,
then the diagonal entries
$A^n(\bz,\bz)=(A^n\dz^{k+1}_\bz,\dz^{k+1}_\bz)_{\ell^2(\scg_k)}\ge0$.
This finishes the proof of \eqref{b.n}.

Then we turn to estimate $\psi^k_\bz$ for all $k\in\zz$
and $\bz\in\scg_k$, where $\psi^k_\bz$ is as in \eqref{2.17x2}
with $\az$ replaced by $\bz$.
From the definition of $\psi^k_\bz$ in \eqref{2.17x2},
\eqref{b.y} and \eqref{b.n},
it follows that
\begin{eqnarray}\label{2.23x1}
\lf|\psi^k_\bz\lf(y^k_\bz\r)\r|&&=
\lf|\sum_{\gz\in\scg_k}\wz{M}^{-1/2}(\bz,\gz)
\frac{s^{k+1}_\gz(y^k_\bz)}{\sqrt{\mu^{k+1}_\gz}}\r|
=\frac{|\wz{M}^{-1/2}(\bz,\bz)|}{\sqrt{\mu^{k+1}_{\bz}}}
\ge \frac{c_3}{\sqrt{\mu^{k+1}_{\bz}}},
\end{eqnarray}
where $\mu^{k+1}_\bz:=V(x^{k+1}_\bz,\dz^{k+1})$.

Moreover, let $\ez_0\in(0,1)$ be a constant which
will be determined later. Recall $y^k_\bz:=x^{k+1}_\bz$
for all $k\in\zz$ and $\bz\in\scg_k$.
By \eqref{b.b}, we know that,
if $x\in B(y^k_\bz,\ez_0\dz^k)\st Q^k_{\az}$
(provided that $\ez_0$ is small enough), then
there exists a positive constant $\wz{C}$ such that, for all
$x\in B(y^k_\bz,\ez_0\dz^k)$,
$$
\lf|\psi^k_\bz(x)-\psi^k_\bz\lf(y^k_\bz\r)\r|
\le \wz{C}\frac1{\sqrt{\mu^{k+1}_\bz}}\ez^{\eta}_0
e^{-\nu\dz^{-k}d(y^k_\bz,x)}
\le\frac{\wz{C}\ez^{\eta}_0}{\sqrt{\mu^{k+1}_\bz}},
$$
which, combined with \eqref{2.23x1}, further implies that,
if we choose $\ez_0$ small enough, then, for all
$x\in B(y^k_\bz,\ez_0\dz^k)$,
\begin{equation}\label{b.o}
\lf|\psi^k_\bz(x)\r|\ge\lf|\psi^k_\bz\lf(y^k_\bz\r)\r|
-\lf|\psi^k_\bz(x)-\psi^k_\bz\lf(y^k_\bz\r)\r|
\ge\frac{c_3-\wz{C}\ez^{\eta}_0}{\sqrt{\mu^{k+1}_\bz}}
\ge\frac{c_3}{2\sqrt{\mu^{k+1}_\bz}}\gls\frac1{\sqrt{\mu^{k+1}_\bz}}.
\end{equation}

Now we are ready to estimate $\psi^k_{\az,\,\bz}$
for all $(k,\az)\in\wz{\sca}$ and
$\bz\in\wz{L}(k,\az)$.
By $y^k_\bz:=x^{k+1}_\bz$, $d(x_\bz^{k+1},x^k_\az)<2\dz^k$
[see Remark \ref{rb.d}(i)] and $B(x^k_\az,\frac13\dz^k)\st Q^k_\az$
[see Theorem \ref{tb.c}(iv)], we have
$$
\mu^{k+1}_\bz\le V\lf(x^{k+1}_\bz,\dz^k\r)
\le V\lf(x^k_\az,3\dz^k\r)
\ls V\lf(x^k_\az,\frac13\dz^k\r)
\ls\mu\lf(Q^k_\az\r).
$$
This, together with \eqref{b.o}, then finishes
the proof of Theorem \ref{tb.m}.
\end{proof}

\begin{remark}\label{rb.z}
Let $(k,\az)\in\wz{\sca}$ with $\wz{\sca}$ as in \eqref{2.19x1},
$\bz\in L(k,\az)$ with $L(k,\az)$ as in \eqref{b.k}, and
$B(y^k_\bz,\ez_0\dz^k)$
be as in Theorem \ref{tb.m}. Now we claim that
$$
V\lf(y^k_\bz,\ez_0\dz^k\r)\sim\mu\lf(Q^k_\az\r).
$$
Indeed, from Remark \ref{rb.d}(i), $(k+1,\bz)\le(k,\az)$,
\eqref{a.b} and $B(y^k_\bz,\ez_0\dz^k)\st Q^k_\az$, it follows that
\begin{eqnarray*}
\mu\lf(Q^k_\az\r)&&\le V\lf(x^k_\az,4\dz^k\r)
\le V\lf(y^{k}_{\bz},6\dz^k\r)\\
&&\ls V\lf(y^{k}_{\bz},\ez_0\dz^{k+1}\r)
\ls V\lf(y^{k}_{\bz},\ez_0\dz^{k}\r)
\ls\mu\lf(Q^k_\az\r),
\end{eqnarray*}
which shows the above claim.
\end{remark}

\section{An Unconditional Basis of $\hona$}\label{s3}

\hskip\parindent In this section, we obtain an unconditional basis of $\hona$.
Now we first recall the following notion of Hardy spaces $\hona$,
which was introduced in \cite{cw77}.

\begin{definition}\label{dc.k}
Let $q\in(1, \fz]$. A function $a$ on
$\cx$ is called  a {\it $(1, q)$-atom} if

(i) $\supp(a)\subset B$ for some ball $B\subset\cx$;

(ii) $\|a\|_{L^q(\cx)}\le [\mu(B)]^{1/q-1}$;

(iii) $\int_\cx a(x)\, d\mu(x)=0$.

A function $f\in L^1(\cx)$ is said to be in the {\it Hardy space
$H_{\rm at}^{1,\,q}(\cx)$} if there exist
$(1,q)$-atoms $\{a_j\}_{j=1}^\fz$ and numbers
$\{\lz_j\}_{j=1}^\fz\subset\cc$ such that
\begin{equation}\label{c.z}
f=\sum_{j=1}^\fz\lz_j a_j,
\end{equation}
which converges in $L^1(\cx)$, and
$$\sum_{j=1}^\fz|\lz_j|<\fz.$$
Moreover, the norm of $f$ in $H_{\rm at}^{1,\,q}(\cx)$
is defined by setting
$$
\|f\|_{H_{\rm at}^{1,\,q}(\cx)}
:=\inf\lf\{\sum_{j\in\nn}|\lz_j|\r\},
$$
where the infimum is taken over all possible decompositions of
$f$ as in \eqref{c.z}.
\end{definition}

Coifman and Weiss \cite{cw77}
proved that $H_{\rm at}^{1,\,q}(\cx)$ and
$H^{1,\,\fz}_{\rm at}(\cx)$ coincide with equivalent norms for all
different $q\in(1, \fz)$.
Thus, from now on, we denote $H^{1,\,q}_{\rm at}(\cx)$
simply by $H^1_{\rm at}(\cx)$.

\begin{remark}\label{rc.i}
It was shown in \cite{cw77} that $\hona$ is a Banach space
which is the predual of $\mathop\mathrm{BMO}(\cx)$.
\end{remark}

We then recall the molecular characterization of $\hona$ from
\cite{hyz}, which plays important roles in establishing
equivalent characterizations of $\hona$ via wavelets, since it
partially compensates the defect of the
regular wavelets without bounded supports.

The following notions of $(1,q,\eta)$-molecules are from \cite{hyz}.

\begin{definition}\label{dc.z}
Let $q\in(1, \fz]$ and $\{\eta_k\}_{k\in\nn}\st[0,\fz)$ satisfy
\begin{equation}\label{c.w}
\sum_{k\in\nn}k\eta_k<\fz.
\end{equation}

A function $m\in\lq$ is called a \emph{$(1,q,\eta)$-molecule centered
at a ball $B:=B(x_0,r)$}, for some $x_0\in\cx$ and $r\in(0,\fz)$, if

(M1) $\|m\|_{\lq}\le[\mu(B)]^{1/q-1}$;

(M2) for all $k\in\nn$,
$$
\lf\|m\chi_{B(x_0,2^kr)\bh B(x_0,2^{k-1}r)}\r\|_{\lq}
\le\eta_k2^{k(1/q-1)}[\mu(B)]^{1/q-1};
$$

(M3) $\int_\cx m(x)\,d\mu(x)=0$.
\end{definition}

Then the following molecular characterization of the space $\hona$
is a slight variant of \cite[Theorem 2.2]{hyz} which is originally related
to the quasi-metric $\rho$ as in \eqref{a.x}
and is obviously true with $\rho$ replaced by $d$.

\begin{theorem}\label{tc.y}
Suppose that $(\cx,d,\mu)$ is a metric measure space of homogeneous type.
Let $q\in(1, \fz]$ and $\eta=\{\eta_k\}_{k\in\nn}\st[0,\fz)$
satisfy \eqref{c.w}.
Then there exists a positive constant
$C$ such that, for any $(1,q,\eta)$-molecule $m$, $m\in\hona$
and
$$\|m\|_{\hona}\le C.$$
Moreover, $f\in\hona$ if and only if there exist
$(1,q,\eta)$-molecules $\{m_j\}_{j\in\nn}$ and numbers
$\{\lz_j\}_{j\in\nn}\st\cc$ such that
$$
f=\sum_{j\in\nn}\lz_jm_j,
$$
which converges in $\lon$. Furthermore,
$$
\|f\|_{\hona}\sim\inf\lf\{\sum_{j\in\nn}|\lz_j|\r\},
$$
where the infimum is taken over all the decompositions of $f$
as above and the equivalent positive constants
are independent of $f$.
\end{theorem}

In order to show that $\{\psi_{\az,\,\bz}^k\}_{(k,\,\az)\in\wz{\sca},\,
\bz\in\wz{L}(k,\,\az)}$
is an unconditional basis of $\hona$, we need some
notions and basic properties
of the unconditional convergence and the unconditional basis from
\cite{lt77,w97}.

\begin{definition}\label{dc.l}
(i) Let $A$ be some countable index set and
 $\{x_n\}_{n\in A}$ a countable family of vectors
in a Banach space $\cb$.
The series $\sum_{n\in A}x_n$ is
said to be \emph{unconditionally convergent}
if, for each permutation $\sz:\ \nn\to A$, namely, a bijection,
the series $\sum_{k=0}^\fz x_{\sz(k)}$ still converges in $\cb$.

(ii) A countable family $\{x_n\}_{n\in A}$ of vectors
in a Banach space $\cb$
is called an \emph{unconditional basis} if, for any $x\in\cb$,
there exists a unique sequence of scalars, $\{\lz_n\}_{n\in A}\st\cc$,
such that
$$
x=\sum_{n\in A}\lz_n x_n\quad {\rm in}\quad \cb
$$
and the expansion $\sum_{n\in A}\lz_n x_n$ of $x$
converges unconditionally.
\end{definition}

\begin{remark}\label{r3.3x}
It was shown in \cite[Proposition 1.c.1]{lt77} that
$\{x_n\}_{n\in A}$ is an unconditional basis
of a Banach space $\cb$ if and only if, for any sequence
$\{\ez_n\}_{n\in A}\st\{-1,1\}$, $\sum_{n\in A}\ez_n x_n$ converges in $\cb$.
\end{remark}

The following useful lemma is a variant
of \cite[Proposition 8.8]{w97} on Euclidean spaces.

\begin{lemma}\label{lc.a}
Suppose that $(\cx,d,\mu)$ is a metric measure space of homogeneous type.
Let $a$ be a $(1,\fz)$-atom.
Then $\sum_{(k,\,\az)\in\wz{\sca},\,
\bz\in\wz{L}(k,\,\az)}(a,\psi_{\az,\,\bz}^k)\psi_{\az,\,\bz}^k$
converges unconditionally in $\hona$. Moreover,
there exists a positive constant $C$, independent of $a$, such that, for all
subsets $\cs\st\{(k,\az,\bz):\ (k,\az)\in\wz{\sca},\,
\bz\in\wz{L}(k,\az)\}=:\mathscr{I}$
with $\wz{\sca}$ and $\wz{L}(k,\az)$ being, respectively, as in
\eqref{2.19x1} and \eqref{2.19x2},
\begin{equation}\label{3.2x}
\lf\|\sum_{(k,\,\az,\,\bz)\in\cs}\lf(a,\psi_{\az,\,\bz}^k\r)
\psi_{\az,\,\bz}^k\r\|_{\hona}\le C.
\end{equation}
\end{lemma}

\begin{proof}
Let $a$ be a $(1,\fz)$-atom supported in the ball
$B:=B(c_B,r_B)$ with $c_B\in\cx$ and $r_B\in(0,\fz)$,
and $N\in\zz$ satisfy $\dz^{N+1}<r_{B}\le \dz^N$.
We first show that $\sum_{(k,\,\az,\,\bz)\in\mathscr{I}}
(a,\psi_{\az,\,\bz}^k)\psi_{\az,\,\bz}^k$
converges unconditionally in $\hona$ and
\eqref{3.2x} holds true for $\cs=\sci$.

Let $\ca:=\{(k,\,\az,\,\bz)\in\mathscr{I}:\ k\le N\}$,
$\cb:=\{(k,\,\az,\,\bz)\in\mathscr{I}:\ k>N,\ x_\bz^{k+1}\not\in2B\}$
and $\ccc:=\{(k,\,\az,\,\bz)\in\mathscr{I}:\ k>N,\ x_\bz^{k+1}\in2B\}$.
Then we write
\begin{eqnarray*}
\sum_{(k,\,\az,\,\bz)\in\mathscr{I}}
\lf(a,\psi_{\az,\,\bz}^k\r)\psi_{\az,\,\bz}^k
&&=\sum_{(k,\,\az,\,\bz)\in\ca}\lf(a,\psi_{\az,\,\bz}^k\r)\psi_{\az,\,\bz}^k
+\sum_{(k,\,\az,\,\bz)\in\cb}\cdots+\sum_{(k,\,\az,\,\bz)\in\ccc}\cdots\\
&&=:\bsz_{\ca}+\bsz_{\cb}+\bsz_{\ccc}.
\end{eqnarray*}

Let $(k,\az,\bz)\in\sci$.
We first claim that $\frac{\psi_{\az,\,\bz}^k}{\sqrt{V(x_\bz^{k+1},\,\dz^k)}}$
is a $(1,2,\eta)$-molecule multiplied by a positive constant
independent of $k$, $\az$ and $\bz$, where
$\eta:=\{\eta_\ell\}_{\ell=1}^\fz$ and
$$
\eta_\ell:=[2\wz{C}_{(\cx)}]^{\ell/2}\exp\{-\gz2^{\ell-1}\}
\ {\rm for\ any}\ \ell\in\nn
$$
with $\wz{C}_{(\cx)}$ as in \eqref{a.b}.

Indeed, by $\|\psi^k_{\az,\,\bz}\|_{\ltw}=1$
[since $\{\psi^k_{\az,\,\bz}\}_{(k,\,\az,\,\bz)\in\sci}$ is an
orthonormal basis of $\ltw$ (see Theorem \ref{tb.a})], we find that
\begin{equation}\label{c.a}
\lf\|\frac{\psi_{\az,\,\bz}^k}{\sqrt{V(x_\bz^{k+1},\dz^k)}}\r\|_{\ltw}
=\lf[V\lf(x_\bz^{k+1},\dz^k\r)\r]^{-1/2}.
\end{equation}

On the other hand, by \eqref{b.a} and \eqref{a.b}, we know that,
for any $\ell\in\nn$,
\begin{eqnarray*}
&&\lf\|\frac{\psi_{\az,\,\bz}^k}{\sqrt{V(x_\bz^{k+1},\dz^k)}}
\chi_{B(x_\bz^{k+1},\,2^\ell\dz^k)\bh
B(x_\bz^{k+1},\,2^{\ell-1}\dz^k)}\r\|_{\ltw}\\
&&\hs\ls\frac1{V(x_\bz^{k+1},\dz^k)}
\lf\{\int_{B(x_\bz^{k+1},\,2^\ell\dz^k)\bh
B(x_\bz^{k+1},\,2^{\ell-1}\dz^k)}e^{-2\nu\dz^{-k}d(x_\bz^{k+1},\,x)}
\,d\mu(x)\r\}^{1/2}\\
&&\hs\ls\frac1{V(x_\bz^{k+1},\dz^k)}e^{-\nu2^{\ell-1}}
\lf[V\lf(x_\bz^{k+1},2^\ell\dz^k\r)\r]^{1/2}\\
&&\hs\ls e^{-\nu2^{\ell-1}}\lf[\wz{C}_{(\cx)}\r]^{\ell/2}
\lf[V\lf(x_\bz^{k+1},\dz^k\r)\r]^{-1/2}
\sim\eta_\ell2^{-\ell/2}\lf[V\lf(x_\bz^{k+1},\dz^k\r)\r]^{-1/2}.
\end{eqnarray*}
This, combined with \eqref{c.a} and
$\sum_{\ell=1}^\fz\ell\eta_\ell<\fz$, implies the above claim.
Moreover, by this claim and Theorem \ref{tc.y}, we conclude that,
for all $(k,\az,\bz)\in\sci$,
$$
\|\psi_{\az,\,\bz}^k\|_{\hona}\ls\sqrt{V(x_\bz^{k+1},\dz^k)},
$$
where the implicit positive constant is independent of
$k$, $\az$ and $\bz$.

In order to estimate $\bsz_\ca$, we first control
$|(a,\psi_{\az,\,\bz}^k)|$ for all $(k,\az,\bz)\in\ca$.
From the vanishing moment of $a$,
\eqref{b.b} and $r_{B}\le\dz^N\le\dz^k$, we deduce that
\begin{eqnarray*}
\lf|\lf(a,\psi_{\az,\,\bz}^k\r)\r|
&&=\lf|\int_{B}a(x)\ov{[\psi_{\az,\,\bz}^k(x)
-\psi_{\az,\,\bz}^k(c_{B})]}\,d\mu(x)\r|\\
&&\le\|a\|_{\li}\int_{B}\lf|\psi_{\az,\,\bz}^k(x)-\psi_{\az,\,\bz}^k(c_{B})\r|
\,d\mu(x)\\
&&\ls\frac1{\mu(B)}\frac1{\sqrt{V(x_\bz^{k+1},\dz^k)}}
\int_{B}\lf[\frac{r_{B}}{\dz^k}\r]^\eta
e^{-\nu\dz^{-k}d(x_\bz^{k+1},\,x)}\,d\mu(x),
\end{eqnarray*}
which, combined with the above claim, Theorem \ref{tc.y},
Lemma \ref{lb.x}, \eqref{a.b} and $\eta\in(0,1]$, implies that
\begin{eqnarray}\label{c.r}
&&\sum_{(k,\,\az,\,\bz)\in\ca}\lf|(a,\psi_{\az,\,\bz}^k)\r|
\lf\|\psi_{\az,\,\bz}^k\r\|_{\hona}\\
&&\noz\hs\ls\sum_{(k,\,\az,\,\bz)\in\ca}
\sqrt{V\lf(x_\bz^{k+1},\dz^k\r)}\lf|(a,\psi_{\az,\,\bz}^k)\r|\\
&&\noz\hs\ls\frac1{\mu(B)}\int_{B}\sum_{\{k\in\zz:\ \dz^k\ge r_{B}\}}
\lf[\frac{r_{B}}{\dz^k}\r]^\eta\sum_{\{\az\in\sca_k,\,
\bz\in\wz{L}(k,\,\az):\ \#L(k,\,\az)>1\}}
e^{-\nu\dz^{-k}d(x_\bz^{k+1},\,x)}\,d\mu(x)\\
&&\noz\hs\ls\frac1{\mu(B)}\int_{B}\sum_{\{k\in\zz:\ \dz^k\ge r_{B}\}}
\lf[\frac{r_{B}}{\dz^k}\r]^\eta
e^{-\nu\dz^{-k}d(x,\,\scy^k)/2}\,d\mu(x)\\
&&\noz\hs\ls\frac1{\mu(B)}\int_{B}\sum_{\{k\in\zz:\ \dz^k\ge r_{B}\}}
\lf[\frac{r_{B}}{\dz^k}\r]^\eta\,d\mu(x)\ls1,
\end{eqnarray}
where, for any $k\in\zz$, $\sca_k$, $L(k,\az)$ and $\wz{L}(k,\az)$
are, respectively, as in \eqref{b.v}, \eqref{b.k} and \eqref{2.19x2},
$
\scy^k:=\scx^{k+1}\bh\scx^k
$
with $\scx^k$ as in \eqref{2.1x}, and
the implicit positive constant is independent of $a$.
Thus, by this and the completion of $\hona$
[see Remark \ref{rc.i}], we know that
$\bsz_{\ca}$ converges unconditionally in $\hona$ and
$\|\bsz_{\ca}\|_{\hona}\ls1$.

Then we estimate $\bsz_\cb$. By the above claim,
Theorem \ref{tc.y}, the size condition of $a$,
\eqref{b.a}, $d(x,x_\bz^{k+1})\ge
\frac12d(x_\bz^{k+1},c_{B})$ for $x\in B$
and $x_\bz^{k+1}\not\in 2B$, Lemma \ref{lb.x}
and $r_B>\dz^{N+1}$,
we conclude that
\begin{eqnarray}\label{c.s}
&&\sum_{(k,\,\az,\,\bz)\in\cb}\lf|(a,\psi_{\az,\,\bz}^k)\r|
\lf\|\psi_{\az,\,\bz}^k\r\|_{\hona}\\
&&\noz\hs\hs\hs\ls\sum_{(k,\,\az,\,\bz)\in\cb}
\sqrt{V\lf(x_\bz^{k+1},\dz^k\r)}\lf|\lf(a,\psi_{\az,\,\bz}^k\r)\r|\\
&&\noz\hs\hs\hs\ls\frac1{\mu(B)}\int_{B}\sum_{k=N+1}^\fz
\sum_{\{\az\in\sca_k,\,\bz\in\wz{L}(k,\,\az):\ x_\bz^{k+1}
\not\in2B,\,\#L(k,\,\az)>1\}}
e^{-\nu\dz^{-k}d(x_\bz^{k+1},\,x)}\,d\mu(x)\\
&&\noz\hs\hs\hs\ls\sum_{k=N+1}^\fz
\sum_{\{\az\in\sca_k,\,\bz\in\wz{L}(k,\,\az):\ x_\bz^{k+1}
\not\in2B,\,\#L(k,\,\az)>1\}}
e^{-2^{-1}\nu\dz^{-k}d(x_\bz^{k+1},\,c_{B})}\\
&&\noz\hs\hs\hs\ls\sum_{k=N+1}^\fz e^{-2^{-2}\nu\dz^{-k}
d(c_{B},\,\scy^k\bh 2B)}
\ls\sum_{k=N+1}^\fz e^{-\nu\dz^{-k}r_{B}}
\ls\sum_{k=N+1}^\fz e^{-\nu\dz^{N-k+1}}\ls1,
\end{eqnarray}
where the implicit positive constant is independent of $a$.
Thus, similar to $\bsz_\ca$, we know that
$\bsz_{\cb}$ converges unconditionally in $\hona$ and
$\|\bsz_{\cb}\|_{\hona}\ls1$.

Finally, we prove that $\bsz_\ccc$ unconditionally
converges in $\hona$. For any $M\in\zz\cap[N,\fz)$,
let
$$
\bsz_\ccc^M:=\sum_{\{k>M,\,\az\in\sca_k,\,
\bz\in\wz{L}(k,\,\az):\ x_\bz^{k+1}\in2B,\,\#L(k,\,\az)>1\}}\ez^k_{\az,\,\bz}
\lf(a,\psi_{\az,\,\bz}^k\r)\psi_{\az,\,\bz}^k,
$$
where $\ez^k_{\az,\,\bz}\in\{-1,1\}$
for any $k>M,\,\az\in\sca_k,\,\bz\in\wz{L}(k,\az)$ with
$x_\bz^{k+1}\in2B$ and $\#L(k,\,\az)>1$.
By Remark \ref{r3.3x}, it suffices to show that
$\|\bsz_\ccc^M\|_{\hona}\ls1$ for all $M\ge N$
and all choices of $\ez^k_{\az,\,\bz}\in\{-1,1\}$,
and $\|\bsz_\ccc^M\|_{\hona}\to0$ as $M\to\fz$.

Without loss of generality, we may assume that $\|\bsz_\ccc^M\|_{\ltw}>0$
for all $M\in\zz\cap[N,\fz)$. Otherwise, we only need to consider
all those $M\in\zz\cap[N,\fz)$ such that $\|\bsz^M_{\ccc}\|_{\ltw}>0$.
From Theorem \ref{tb.a} and Definition \ref{dc.k}(ii),
it follows that, for any $M\in\zz$,
\begin{equation}\label{c.4}
\lf\|\bsz^M_\ccc\r\|_{\ltw}\le\lf\|\bsz_\ccc\r\|_{\ltw}
\le\|a\|_{\ltw}\le[\mu(B)]^{-1/2}
\end{equation}
and $\|\bsz^M_\ccc\|_{\ltw}\to0$ as $M\to\fz$.

Let $\mu_M:=\|\bsz^M_\ccc\|_{\ltw}[\mu(4B)]^{1/2}$ for all
$M\in\zz\cap[N,\fz)$. Now we claim that
\begin{equation}\label{3.7x}
\wz{\bsz}^M_\ccc:=\bsz^M_\ccc/\mu_M\quad \mathrm{is\ a}
\quad (1,2,\eta)-\mathrm{molecule},
\end{equation}
centered at ball $4B$, multiplied by some positive constant,
where $\eta:=\{\eta_\ell\}_{\ell=0}^\fz\st[0,\fz)$
and $\eta_\ell:=[2\wz{C}_{(\cx)}]^{\ell/2}2^{-(\ell+1)K_0}/\mu_M$
for some large positive
integer $K_0$ such that $K_0\ge G_0+n+1$, with $G_0$ and $n$
respectively as in Remark \ref{rb.l}(ii) and \eqref{a.b}, and
$$
\sum_{\ell=1}^\fz\ell2^{n\ell/2}\lf[2\wz{C}_{(\cx)}\r]^{\ell/2}
2^{-(\ell+1)K_0}/\mu_M<\fz.
$$
Obviously,
\begin{equation}\label{c.3}
\lf\|\wz{\bsz}^M_\ccc\r\|_{\ltw}
=\lf\|\bsz^M_\ccc\r\|_{\ltw}/\mu_M=[\mu(4B)]^{-1/2}.
\end{equation}

On the other hand, by \eqref{a.b},
we observe that, for any $r_0,\,\nu_0\in(0,\fz)$
and $x_0\in\cx$,
\begin{eqnarray}\label{3.9x1}
&&\int_{\cx}e^{-\nu_0 d(x,\,x_0)/r_0}\,d\mu(x)\\
&&\noz\hs\ls\int_{B(x_0,\,r_0)}e^{-\nu_0 d(x,\,x_0)/r_0}\,d\mu(x)
+\sum_{\ell=1}^\fz\int_{B(x_0,\,(\ell+1)r_0)
\bh B(x_0,\,\ell r_0)}\cdots\\
&&\noz\hs\ls V(x_0,r_0)+\sum_{\ell=1}^\fz e^{-\nu_0\ell}
V\lf(x_0,[\ell+1]r_0\r)\\
&&\noz\hs\ls V(x_0,r_0)+\sum_{\ell=1}^\fz e^{-\nu_0\ell}
(\ell+1)^nV\lf(x_0,r_0\r)\ls V(x_0,r_0).
\end{eqnarray}
From \eqref{3.9x1} and \eqref{b.a}, we further deduce that,
for all $(k,\az,\bz)\in\sci$,
\begin{equation}\label{3.9x2}
\lf\|\psi^k_{\az,\,\bz}\r\|_{\lon}\ls\sqrt{V(x^{k+1}_\bz,\dz^k)}.
\end{equation}

Moreover, for any $\ell\in\zz_+:=\{0\}\cup\nn$ and
$x\in2^{\ell+3}B\bh2^{\ell+2} B$, by \eqref{3.9x2},
\eqref{b.a}, $x^{k+1}_\bz\in2B$ [and hence
$d(x,x_\bz^{k+1})\ge\frac{1}{2}d(x,c_B)$ and
$B(x^{k+1}_\bz, \dz^k)\st3B$], the geometrically doubling condition,
(i) and (iii) of Remark \ref{rb.l}, $K_0\ge G_0+n+1$
and $\dz^{M+1}<r_B$, we conclude that
\begin{eqnarray*}
\lf|\bsz^M_\ccc(x)\r|&&\le\sum_{k=M+1}^\fz\sum_{\{\az\in\sca_k,\,
\bz\in\wz{L}(k,\,\az):\ x_\bz^{k+1}\in2B,\,\#L(k,\,\az)>1\}}
\lf|\lf(a,\psi_{\az,\,\bz}^k\r)\r|\lf|\psi_{\az,\,\bz}^k(x)\r|\\
&&\le\|a\|_{\li}\sum_{k=M+1}^\fz\sum_{\{\az\in\sca_k,\,
\bz\in\wz{L}(k,\,\az):\ x_\bz^{k+1}\in2B,\,\#L(k,\,\az)>1\}}
\lf\|\psi_{\az,\,\bz}^k\r\|_{\lon}\lf|\psi_{\az,\,\bz}^k(x)\r|\\
&&\ls\|a\|_{\li}\sum_{k=M+1}^\fz\sum_{\{\az\in\sca_k,\,
\bz\in\wz{L}(k,\,\az):\ x_\bz^{k+1}\in2B,\,\#L(k,\,\az)>1\}}
e^{-\nu\dz^{-k}d(x,\,x^{k+1}_\bz)}\\
&&\ls\|a\|_{\li}\sum_{k=M+1}^\fz\sum_{\{\az\in\sca_k,\,
\bz\in\wz{L}(k,\,\az):\ x_\bz^{k+1}\in2B,\,\#L(k,\,\az)>1\}}
e^{-\frac{\nu}{2}\dz^{-k}d(x,\,c_B)}\\
&&\ls[\mu(B)]^{-1}\sum_{k=M+1}^\fz\sum_{\{\az\in\sca_k,\,
\bz\in\wz{L}(k,\,\az):\ x_\bz^{k+1}\in2B,\,\#L(k,\,\az)>1\}}
e^{-\nu2^{\ell+1}\frac{r_B}{\dz^{k}}}\\
&&\ls[\mu(B)]^{-1}\sum_{k=M+1}^\fz
\lf[\frac{r_B}{\dz^k}\r]^{G_0}
e^{-\nu2^{\ell+1}\frac{r_B}{\dz^{k}}}\\
&&\ls[\mu(B)]^{-1}\sum_{k=M+1}^\fz
2^{-(\ell+1)K_0}\lf[\frac{\dz^k}{r_B}\r]^{K_0}
\lf[\frac{r_B}{\dz^k}\r]^{G_0}\\
&&\ls[\mu(B)]^{-1}2^{-(\ell+1)K_0}\sum_{k=M+1}^\fz
\frac{\dz^k}{r_B}\ls[\mu(B)]^{-1}2^{-(\ell+1)K_0}.
\end{eqnarray*}
Thus, by this and \eqref{a.b}, we further have
\begin{eqnarray}\label{c.5}
&&\lf\|\wz{\bsz}^M_\ccc\chi_{2^{\ell+3}B\bh2^{\ell+2}B}\r\|_{\ltw}\\
&&\noz\hs\ls\frac1{\mu_M}[\mu(B)]^{-1}2^{-(\ell+1)K_0}
\lf[\wz{C}_{(\cx)}\r]^{\ell/2}[\mu(B)]^{1/2}
\ls2^{-\ell/2}\eta_{\ell}[\mu(4B)]^{-1/2}.
\end{eqnarray}

To prove the claim in \eqref{3.7x}, we need to further show that
\begin{equation}\label{c.o}
\int_{\cx}\wz{\bsz}_{\ccc}^Md\mu=0.
\end{equation}
By the H\"older inequality,
\eqref{c.3}, \eqref{c.5}, \eqref{a.b} and $K_0\ge G_0+n+1$,
we know that
\begin{eqnarray*}
\lf\|\wz{\bsz}_C^M\r\|_{\lon}&&\le\lf\|\wz{\bsz}_C^M\chi_{4B}\r\|_{\lon}
+\sum_{\ell=0}^\fz\lf\|\wz{\bsz}_C^M
\chi_{2^{\ell+3}B\bh2^{\ell+2}B}\r\|_{\lon}\\
&&\le\lf\|\wz{\bsz}_C^M\r\|_{\ltw}[\mu(4B)]^{1/2}
+\sum_{\ell=0}^\fz\lf\|\wz{\bsz}_C^M
\chi_{2^{\ell+3}B\bh2^{\ell+2}B}\r\|_{\ltw}
\lf[\mu\lf(2^{\ell+3}B\r)\r]^{1/2}\\
&&\ls1+\sum_{\ell=0}^\fz 2^{\frac{n}{2}\ell}2^{-\frac{\ell}{2}}\eta_\ell
\ls1.
\end{eqnarray*}

Moreover, let $U_\ell(B):=2^{\ell+3}B\bh2^{\ell+2}B$
for any $\ell\in\zz_+$. By $\bsz_{\ccc}^M\in\lon$,
Theorem \ref{tb.a} and \eqref{b.c}, we conclude that
\begin{eqnarray}\label{c.x}
\qquad\int_{\cx}\bsz_{\ccc}^M d\mu
&&=\int_{\cx}\chi_{4B}\bsz_{\ccc}^M d\mu
+\sum_{\ell=0}^\fz\int_{\cx}\chi_{U_\ell(B)}\bsz_{\ccc}^M d\mu\\
&&\noz=\lf(\bsz_{\ccc}^M,\chi_{4B}\r)+\sum_{\ell=0}^\fz
\lf(\bsz_{\ccc}^M,\chi_{U_\ell(B)}\r)\\
&&\noz=\sum_{k=M+1}^\fz\sum_{\{\az\in\sca_k,\,
\bz\in\wz{L}(k,\,\az):\ x_\bz^{k+1}\in2B,\,\#L(k,\,\az)>1\}}
\ez^{k}_{\az,\,\bz}\lf(a,\psi^k_{\az,\,\bz}\r)
\lf(\psi_{\az,\,\bz}^k,\chi_{4B}\r)\\
&&\noz\hs+\sum_{\ell=0}^\fz\sum_{k=M+1}^\fz\sum_{\{\az\in\sca_k,\,
\bz\in\wz{L}(k,\,\az):\ x_\bz^{k+1}\in2B,\,\#L(k,\,\az)>1\}}
\ez^{k}_{\az,\,\bz}\lf(a,\psi^k_{\az,\,\bz}\r)\\
&&\hs\hs\times\lf(\psi_{\az,\,\bz}^k,\chi_{U_\ell(B)}\r).\noz
\end{eqnarray}

Now we show that
\begin{equation}\label{3.10x}
\sum_{\ell=0}^\fz\sum_{k=M+1}^\fz\sum_{\{\az\in\sca_k,\,
\bz\in\wz{L}(k,\,\az):\ x_\bz^{k+1}\in2B,\,\#L(k,\,\az)>1\}}
\lf|\lf(a,\psi^k_{\az,\,\bz}\r)
\lf(\psi_{\az,\,\bz}^k,\chi_{U_\ell(B)}\r)\r|<\fz.
\end{equation}

Indeed, from the H\"older inequality and Theorem \ref{tb.a},
it follows that
\begin{eqnarray*}
&&\sum_{\ell=0}^\fz\sum_{k=M+1}^\fz\sum_{\{\az\in\sca_k,\,
\bz\in\wz{L}(k,\,\az):\ x_\bz^{k+1}\in2B,\,\#L(k,\,\az)>1\}}
\lf|\lf(a,\psi^k_{\az,\,\bz}\r)
\lf(\psi_{\az,\,\bz}^k,\chi_{U_\ell(B)}\r)\r|\\
&&\hs\le\sum_{\ell=0}^\fz\lf[\sum_{k=M+1}^\fz\sum_{\{\az\in\sca_k,\,
\bz\in\wz{L}(k,\,\az):\ x_\bz^{k+1}\in2B,\,\#L(k,\,\az)>1\}}
\lf|\lf(a,\psi^k_{\az,\,\bz}\r)\r|^2\r]^{1/2}\\
&&\hs\hs\times\lf[\sum_{k=M+1}^\fz\sum_{\{\az\in\sca_k,\,
\bz\in\wz{L}(k,\,\az):\ x_\bz^{k+1}\in2B,\,\#L(k,\,\az)>1\}}
\lf|\lf(\psi_{\az,\,\bz}^k,\chi_{U_\ell(B)}\r)\r|^2\r]^{1/2}\\
&&\hs\le\|a\|_{\ltw}\sum_{\ell=0}^\fz
\lf[\sum_{k=M+1}^\fz\sum_{\{\az\in\sca_k,\,
\bz\in\wz{L}(k,\,\az):\ x_\bz^{k+1}\in2B,\,\#L(k,\,\az)>1\}}
\lf|\lf(\psi_{\az,\,\bz}^k,\chi_{U_\ell(B)}\r)\r|^2\r]^{1/2}.
\end{eqnarray*}

Now we estimate $|(\psi_{\az,\,\bz}^k,\chi_{U_\ell(B)})|$
for any $\ell\in\zz_+$, $k\in\zz\cap[M+1,\fz)$,
$\az\in\sca_k$ and $\bz\in\wz{L}(k,\az)$ with $x^{k+1}_\bz\in2B$
and $\#L(k,\az)>1$.
Indeed, we choose $M_0$ to be a large enough positive constant
such that $M_0\ge G_0+1$ with $G_0$ as in Remark \ref{rb.l}(ii).
From \eqref{b.a} [together with \eqref{3.9x1}], \eqref{a.b},
$\dz^k\le\dz^{M+1}\le\dz^{N+1}<r_B$ for all $k\in\zz\cap[M+1,\fz)$
and $B(x^{k+1}_\bz,\dz^k)\st B(x^{k+1}_\bz,r_B)\st3B$, we deduce that
\begin{eqnarray*}
\lf|\lf(\psi_{\az,\,\bz}^k,\chi_{U_\ell(B)}\r)\r|
&&\ls\frac1{\sqrt{V(x^{k+1}_\bz,\dz^k)}}\int_{U_\ell(B)}
e^{-\nu\dz^{-k}d(x,x^{k+1}_\bz)}\,d\mu(x)\\
&&\ls \frac1{\sqrt{V(x^{k+1}_\bz,\dz^k)}}e^{\frac{\nu}{2}\dz^{-k}
d(c_B,\,x^{k+1}_\bz)}
\int_{U_\ell(B)}e^{-\frac{\nu}{2}\dz^{-k}d(x,c_B)}
e^{-\frac{\nu}{2}\dz^{-k}d(x,x^{k+1}_\bz)}\,d\mu(x)\\
&&\ls e^{\nu\dz^{-k}r_B}e^{-\nu2^{\ell+1}\dz^{-k}r_B}
\sqrt{V(x^{k+1}_\bz,\dz^k)}\\
&&\ls e^{-\nu2^\ell\dz^{-k}r_B}
\sqrt{\mu(B)}
\ls 2^{-\ell M_0}\lf[\frac{\dz^k}{r_B}\r]^{M_0}
\sqrt{\mu(B)},
\end{eqnarray*}
which, combined with the elementary inequality
\begin{equation}\label{x.x}
\lf[\sum_{j=0}^\fz\lf|a_j\r|\r]^p\le\sum_{j=0}^\fz\lf|a_j\r|^p
\quad {\rm for\ all\ }\{a_j\}_{j=0}^{\fz}\st\cc\ {\rm and\ }
p\in(0,1],
\end{equation}
and the fact that
\begin{eqnarray*}
&&\#\lf\{\az\in\sca_k,\,
\bz\in\wz{L}(k,\,\az):\ x_\bz^{k+1}\in2B,\,\#L(k,\,\az)>1\r\}\\
&&\hs=\#\lf\{\bz\in\scg_k:\ x_\bz^{k+1}\in2B\r\}
\ls\lf[\frac{r_B}{\dz^k}\r]^{G_0}
\end{eqnarray*}
[see Remark \ref{rb.l}(ii)], further implies that
\begin{eqnarray*}
&&\sum_{\ell=0}^\fz\sum_{k=M+1}^\fz\sum_{\{\az\in\sca_k,\,
\bz\in\wz{L}(k,\,\az):\ x_\bz^{k+1}\in2B,\,\#L(k,\,\az)>1\}}
\lf|\lf(a,\psi^k_{\az,\,\bz}\r)
\lf(\psi_{\az,\,\bz}^k,\chi_{U_\ell(B)}\r)\r|\\
&&\hs\ls[\mu(B)]^{-1/2}\sum_{\ell=0}^\fz 2^{-\ell M_0}
\sum_{k=M+1}^\fz\lf[\frac{r_B}{\dz^k}\r]^{G_0}
\lf[\frac{\dz^k}{r_B}\r]^{M_0}[\mu(B)]^{1/2}\\
&&\hs\ls\sum_{\ell=0}^\fz 2^{-\ell M_0}
\sum_{k=M+1}^\fz\frac{\dz^k}{r_B}\ls1.
\end{eqnarray*}
This shows \eqref{3.10x}.

From \eqref{c.x}, \eqref{3.10x} and \eqref{b.c},
we further deduce that
\begin{eqnarray*}
\int_{\cx}\bsz_{\ccc}^M d\mu
&&=\sum_{k=M+1}^\fz\sum_{\{\az\in\sca_k,\,
\bz\in\wz{L}(k,\,\az):\ x_\bz^{k+1}\in2B,\,\#L(k,\,\az)>1\}}
\ez^{k}_{\az,\,\bz}\lf(a,\psi^k_{\az,\,\bz}\r)\\
&&\hs\times\lf[\lf(\psi_{\az,\,\bz}^k,\chi_{4B}\r)+\sum_{\ell=0}^\fz
\lf(\psi_{\az,\,\bz}^k,\chi_{U_\ell(B)}\r)\r]\\
&&=\sum_{k=M+1}^\fz\sum_{\{\az\in\sca_k,\,
\bz\in\wz{L}(k,\,\az):\ x_\bz^{k+1}\in2B,\,\#L(k,\,\az)>1\}}
\ez^{k}_{\az,\,\bz}\lf(a,\psi^k_{\az,\,\bz}\r)
\lf(\psi_{\az,\,\bz}^k,1\r)=0,
\end{eqnarray*}
which shows \eqref{c.o} and hence
completes the proof of the above claim in \eqref{3.7x}.

From the above claim,  Theorem \ref{tc.y}, \eqref{c.4}
and the fact that $\|\bsz^M_\ccc\|_{\ltw}\to0$ as $M\to\fz$,
we further deduce that, for all integer $M\ge N$,
$$
\lf\|\bsz^M_\ccc\r\|_{\hona}\ls\mu_M\ls1\quad
{\rm and}\quad \|\bsz^M_\ccc\|_{\hona}\to0\ {\rm as}\ M\to\fz.
$$
This, combined with \eqref{c.r} and \eqref{c.s},
shows that $\sum_{(k,\,\az,\,\bz)\in\mathscr{I}}
(a,\psi_{\az,\,\bz}^k)\psi_{\az,\,\bz}^k$
converges unconditionally in $\hona$ and
\eqref{3.2x} holds true for $\cs=\sci$.

By the above proof of \eqref{3.2x} with $\cs=\sci$,
we easily see that \eqref{3.2x} also holds
true for any subset $\cs\st\sci$, which
completes the proof of Lemma \ref{lc.a}.
\end{proof}

To obtain an unconditional basis of $\hona$,
we need the boundedness of Calder\'on-Zygmund operators
from $\hona$ to $\lon$ and from $\hona$ to itself.
We first recall some notions and notation from \cite{cw71};
see also \cite{ah13}.
Let $s\in(0,\eta]$ with $\eta$ as in \eqref{b.1}
and $C^s_b(\cx)$ be the set of all
\emph{$s$-H\"older continuous} functions $f$ [namely,
$\sup_{\{x,\,y\in\cx:\ x\neq y\}}
\frac{|f(x)-f(y)|}{[d(x,y)]^s}<\fz$]
with bounded supports, whose \emph{dual space} is denoted
by $(C^s_b(\cx))^*$. We point out that $C^s_b(\cx)$ is dense in $\ltw$
(see, for example, \cite[Proposition 4.5]{ah13}).

Now we introduce the notion of Calder\'on-Zygmund
operators from \cite{cw71}; see also \cite{ah13}.

\begin{definition}\label{d4.6}
A function $K\in L_\loc^1(\{\cx\times
\cx\}\bh\{(x,x):x\in\cx\})$ is called a \emph{Calder\'on-Zygmund
kernel} if there exists a positive constant $C_{(K)}$,
depending on $K$, such that
\begin{enumerate}
\item[(i)] for all $x,\,y\in\cx$ with $x\ne y$,
\begin{equation}\label{d.a}
|K(x,y)|\le C_{(K)}\frac{1}{V(x,y)};
\end{equation}

\item[(ii)] there exist positive constants
$s\in (0,1]$ and $c_{(K)}\in(0,1)$, depending on $K$,
such that
\begin{enumerate}
\item[${\rm (ii)}_1$] for all $x,\,\wz x,\,y\in\cx$
with $d(x,y)\ge c_{(K)}d(x,\wz{x})>0$,
\begin{equation}\label{d.b}
|K(x,y)-K(\wz{x},y)|\le C_{(K)}
\lf[\frac{d(x,\wz{x})}{d(x,y)}\r]^s
\frac1{V(x,y)};
\end{equation}

\item[${\rm (ii)}_2$] for all $x,\,\wz x,\,y\in\cx$
with $d(x,y)\ge c_{(K)}d(y,\wz{y})>0$,
\begin{equation}\label{d.g}
\lf|K(x,y)-K(x,\wz{y})\r|\le C_{(K)}
\lf[\frac{d(y,\wz{y})}{d(x,y)}\r]^s
\frac1{V(x,y)}.
\end{equation}
\end{enumerate}
\end{enumerate}

Let $T:\ C^s_b(\cx)\to (C^s_b(\cx))^*$ be a linear continuous
operator. The operator $T$ is
called  a \emph{Calder\'on-Zygmund operator}
with kernel $K$ satisfying \eqref{d.a}, \eqref{d.b}
and \eqref{d.g} if, for
all $f\in C^s_b(\cx)$,
\begin{equation}\label{d.c}
Tf(x):=\int_{\cx}K(x,y)f(y)\,d\mu(y),\quad x\not\in\supp(f).
\end{equation}
\end{definition}

Then we recall some results from
\cite[Proposition 3.1]{yz08}
(see also \cite[Theorem 4.2]{hyz}) about the
boundedness of Calder\'on-Zygmund operators. In what follows
$T^*1=0$ means that, for all $(1,2)$-atom $a$,
$\int_\cx Ta(x)\,d\mu(x)=0$. By some careful
examinations, we see that this result remains valid over
the metric measure space of homogeneous type
without resorting to the reverse doubling condition,
the details being omitted.

\begin{theorem}\label{te.g}
Let $(\cx,d,\mu)$ be a metric measure space of homogeneous type.
Suppose that $T$ is a Calder\'on-Zygmund operator as in
\eqref{d.c} which is bounded on $\ltw$.
\begin{enumerate}
\item[\rm (i)] Then there exists a positive constant
$C$, depending only on $\|T\|_{\cl(\ltw)}$, $s$, $C_{(K)}$,
$c_{(K)}$ and $\wz{C}_{(\cx)}$,
 such that, for all $f\in \hona$, $Tf\in \lon$ and
$\|Tf\|_{\lon}\le C\|f\|_{\hona}$.

\item[\rm (ii)] If further assuming that $T^*1=0$, then
there exists a positive constant $\wz{C}$,
depending only on $\|T\|_{\cl(\ltw)}$, $s$, $C_{(K)}$,
$c_{(K)}$ and $\wz{C}_{(\cx)}$,
such that, for all $f\in \hona$, $Tf\in \hona$ and
$\|Tf\|_{\hona}\le C\|f\|_{\hona}$.
\end{enumerate}
\end{theorem}

Now we show the following conclusion on an unconditional basis of $\hona$.

\begin{theorem}\label{tc.b}
Let $(\cx,d,\mu)$ be a metric measure space of homogeneous type.
Then
$$
\{\psi_{\az,\,\bz}^k\}_{(k,\,\az,\,\bz)\in\sci},
$$
with $\sci$ as in Lemma \ref{lc.a}, is an unconditional basis of $\hona$.
\end{theorem}

\begin{proof}
We first show that, for any $(1,\fz)$-atom $a$
\begin{equation}\label{c.b}
a=\sum_{(k,\,\az,\,\bz)\in\sci}
\lf(a,\psi_{\az,\,\bz}^k\r)\psi_{\az,\,\bz}^k
\quad {\rm in}\quad \hona.
\end{equation}
Observe that, by Lemma \ref{lc.a}, $\sum_{(k,\,\az,\,\bz)\in\sci}
\lf(a,\psi_{\az,\,\bz}^k\r)\psi_{\az,\,\bz}^k$ converges
unconditionally in $\hona$.

Let
\begin{equation}\label{3.24z}
\lf\{\sci_N:\ N\in\nn,\ \sci_N\st\sci\ {\rm and\ \sci_N\ is\ finite}\r\}
\end{equation}
be any collection satisfy $\sci_N\uparrow\sci$ (namely, for any $N\in\nn$,
$\sci_N\st\sci_{N+1}$ and $\sci=\bigcup_{N\in\nn}\sci_N$) and
$$
S_N(a):=\sum_{(k,\,\az,\,\bz)\in\sci_N}
\lf(a,\psi_{\az,\,\bz}^k\r)\psi_{\az,\,\bz}^k.
$$
By Lemma \ref{lc.a}, we conclude that there exists $\wz{a}\in\hona$
such that
\begin{equation}\label{c.c}
\wz{a}=\lim_{N\to\fz}\sum_{(k,\,\az,\,\bz)\in\sci_N}
\lf(a,\psi_{\az,\,\bz}^k\r)\psi_{\az,\,\bz}^k
\quad {\rm in}\quad \hona,
\end{equation}
which, together with $\hona\st\lon$ and the Riesz lemma, further
implies that there exists a subsequence
$
\{\sum_{(k,\,\az,\,\bz)\in\sci_{N_m}}
(a,\psi_{\az,\,\bz}^k)\psi_{\az,\,\bz}^k\}_{m\in\nn}
$
of
$
\{\sum_{(k,\,\az,\,\bz)\in\sci_N}
(a,\psi_{\az,\,\bz}^k)\psi_{\az,\,\bz}^k\}_{N\in\nn}
$
such that
\begin{equation}\label{c.cx}
\wz{a}=\lim_{m\to\fz}\sum_{(k,\,\az,\,\bz)\in\sci_{N_m}}
\lf(a,\psi_{\az,\,\bz}^k\r)\psi_{\az,\,\bz}^k
\quad \mu-{\rm almost\ everywhere\ on\ }\cx.
\end{equation}

On the other hand, from Theorem \ref{tb.a},
together with $a\in\ltw$, it follows that
$$
a=\lim_{m\to\fz}\sum_{(k,\,\az,\,\bz)\in\sci_{N_m}}
\lf(a,\psi_{\az,\,\bz}^k\r)\psi_{\az,\,\bz}^k\quad
{\rm in}\quad \ltw,
$$
which, combined with the Riesz lemma and \eqref{c.cx}, further
implies that
$$
\wz{a}=a\quad \mu-{\rm almost\ everywhere\ on\ }\cx.
$$
This, together with \eqref{c.c}, then finishes the proof of \eqref{c.b}.

For all $(k,\az,\bz)\in\sci$ with $\sci$ as in Lemma \ref{lc.a}, from
$\psi^k_{\az,\,\bz}\in\li\st\bmo$, it follows that
$$
\lf\langle f,\psi_{\az,\,\bz}^k\r\rangle
:=\int_{\cx}f\psi_{\az,\,\bz}^k\,d\mu
$$
is well defined in the sense of duality
between $\hona$ and $\bmo$.

Then we claim that, for any $f\in\hona$,
\begin{equation}\label{3.22y}
f=\sum_{(k,\,\az,\,\bz)\in\sci}
\lf\langle f,\psi_{\az,\,\bz}^k\r\rangle\psi_{\az,\,\bz}^k\quad
{\rm in}\quad \hona.
\end{equation}
By Definition \ref{dc.k}, we see that there exist
sequences $\{a_j\}_{j\in\nn}$ of $(1,\fz)$-atoms
and numbers $\{\lz_j\}_{j\in\nn}\st\cc$
satisfying $f=\sum_{j=1}^\fz\lz_j a_j$ in $\lon$
and $\sum_{j=1}^\fz|\lz_j|\ls\|f\|_{\hona}$.

From \eqref{c.b}, it follows that, for any $M\in\nn$,
$f_M:=\sum_{j=1}^M\lz_j a_j$ satisfies
\begin{equation}\label{3.22z}
f_M=\sum_{(k,\,\az,\,\bz)\in\sci}
\lf(f_M,\psi_{\az,\,\bz}^k\r)\psi_{\az,\,\bz}^k\quad
{\rm in}\quad \hona.
\end{equation}

Let $N\in\nn$ and, for any suitable function $f$,
$$
S_N(f):=\sum_{(k,\,\az,\,\bz)\in\sci_N}
\lf\langle f,\psi_{\az,\,\bz}^k\r\rangle\psi_{\az,\,\bz}^k\quad
{\rm with}\ \sci_N\ {\rm as\ in}\ \eqref{3.24z}.
$$
Then, by \eqref{3.22z}, we see that, for any fixed $M\in\nn$,
\begin{equation}\label{c.6}
\lim_{N\to\fz}\lf\|S_N\lf(f_M\r)-f_M\r\|_{\hona}=0.
\end{equation}

Observe that, for any $N\in\nn$, $S_N$ and $S^*_N$,
where $S^*_N$ denotes the \emph{adjoint operator} of $S_N$,
are integral operators with kernels
$$
K_N(x,y):=\sum_{(k,\,\az,\,\bz)\in\sci_N}
\psi^k_{\az,\,\bz}(x)\psi^k_{\az,\,\bz}(y)
$$
and $K^*_N(x,y):=K_N(y,x)$
for all $x,\,y\in\cx$ with $x\neq y$, respectively.
By \cite[Proposition 10.3]{ah13}, we know that, for each $N\in\nn$,
$K_N$ satisfies \eqref{d.a}, \eqref{d.b} and \eqref{d.g}.
From Theorem \ref{tb.a}, we deduce that
$S^*_N(1)=0$ and $\|S_N(f)\|_{\ltw}\le\|f\|_{\ltw}$
for all $f\in\ltw$.

By this and Theorem \ref{te.g}(ii), we conclude that
$\{S_N\}_{N\in\nn}$ are bounded on $\hona$ uniformly
in $N\in\nn$, which further implies that, for each $N\in\nn$,
\begin{equation}\label{3.28x}
\lf\|S_N\lf(f_M\r)-S_N(f)\r\|_{\hona}
=\lf\|S_N\lf(f_M-f\r)\r\|_{\hona}\ls\|f_M-f\|_{\hona}.
\end{equation}
This, combined with \eqref{c.6}, further implies that
\begin{eqnarray*}
&&\limsup_{N\to\fz}\lf\|S_N(f)-f\r\|_{\hona}\\
&&\hs\le\limsup_{N\to\fz}
\lf[\lf\|S_N(f)-S_N(f_M)\r\|_{\hona}+\lf\|S_N(f_M)-f_M\r\|_{\hona}
+\lf\|f_M-f\r\|_{\hona}\r]\\
&&\hs\ls\lf\|f_M-f\r\|_{\hona}
+\lim_{N\to\fz}\lf\|S_N(f_M)-f_M\r\|_{\hona}\\
&&\hs\sim\lf\|f_M-f\r\|_{\hona}\to0,\quad {\rm as\ }M\to\fz,
\end{eqnarray*}
which completes the proof of the claim \eqref{3.22y}.

Now we show the uniqueness of the representations
$$
f=\sum_{(k,\,\az,\,\bz)\in\sci}
\lz_{\az,\,\bz}^k\psi_{\az,\,\bz}^k\quad
{\rm in}\quad \hona
$$
for all numbers $\{\lz_{\az,\,\bz}^k\}_{(k,\,\az,\,\bz)\in\sci}\st\cc$.
Indeed, by the fact that, for all $(k,\az,\bz)\in\sci$,
$\psi^k_{\az,\,\bz}\in\li\st\bmo$ and the
orthogonality of $\{\psi_{\az,\,\bz}^k\}_{(k,\,\az,\,\bz)\in\sci}$
(see Theorem \ref{tb.a}), we conclude that,
for all $(\ell,\gz,\tz)\in\sci$,
$$
\lf\langle f,\psi_{\gz,\,\tz}^\ell\r\rangle
=\sum_{(k,\,\az,\,\bz)\in\sci}
\lz_{\az,\,\bz}^k\lf(\psi_{\az,\,\bz}^k,\psi_{\gz,\,\tz}^\ell\r)
=\lz_{\gz,\,\tz}^\ell,
$$
which is the desired result.

Finally, we prove that $\sum_{(k,\,\az,\,\bz)\in\sci}
\langle f,\psi_{\az,\,\bz}^k\rangle\psi_{\az,\,\bz}^k$ converges
unconditionally. By Remark \ref{r3.3x}, we know that
it suffices to show that, for any
sequence $\{\ez^k_{\az,\,\bz}\}_{(k,\,\az,\,\bz)\in\sci}\st\{-1,1\}$,
\begin{equation}\label{3.22x}
\sum_{(k,\,\az,\,\bz)\in\sci}\ez^k_{\az,\,\bz}
\lf\langle f,\psi_{\az,\,\bz}^k\r\rangle\psi_{\az,\,\bz}^k\quad
{\rm converges\ in\ }\hona.
\end{equation}

Let $N\in\nn$ and
$$
\wz{S}_N(f):=\sum_{(k,\,\az,\,\bz)\in\sci_N}
\ez^k_{\az,\,\bz}
\lf\langle f,\psi_{\az,\,\bz}^k\r\rangle\psi_{\az,\,\bz}^k
\quad{\rm with}\ \sci_N\ {\rm as\ in}\ \eqref{3.24z}.
$$
By some arguments similar to those used in \eqref{3.28x},
we conclude that $\wz{S}_N$ is bounded on $\hona$ uniformly
in $N\in\nn$ and hence, for any $N,\,M\in\nn$, if $f_M$ is as above, then
$$
\lf\|\wz{S}_N(f)-\wz{S}_N(f_M)\r\|_{\hona}
\ls\lf\|f-f_M\r\|_{\hona}.
$$
Observe also that, by Lemma \ref{lc.a} and Remark \ref{r3.3x},
we know that $\{\wz{S}_N\lf(a_j\r)\}_{N\in\nn}$ for $j\in\{1,\ldots, M\}$
is a Cauchy sequence in $\hona$. By these facts, we further conclude
that, for all $M\in\nn$,
\begin{eqnarray*}
&&\limsup_{N,\,K\to\fz}\lf\|\wz{S}_N(f)-\wz{S}_K(f)\r\|_{\hona}\\
&&\hs\le\limsup_{N\to\fz}\lf\|\wz{S}_N(f)-\wz{S}_N\lf(f_M\r)\r\|_{\hona}
+\limsup_{N,\,K\to\fz}\lf\|\wz{S}_N\lf(f_M\r)-\wz{S}_K\lf(f_M\r)\r\|_{\hona}\\
&&\hs\hs+\limsup_{K\to\fz}\lf\|\wz{S}_K\lf(f_M\r)-\wz{S}_K(f)\r\|_{\hona}\\
&&\hs\ls\lf\|f-f_M\r\|_{\hona}+\sum_{j=1}^M\lf|\lz_j\r|
\lim_{N,\,K\to\fz}\lf\|\wz{S}_N\lf(a_j\r)-\wz{S}_K\lf(a_j\r)\r\|_{\hona}\\
&&\hs\ls\lf\|f-f_M\r\|_{\hona}\to 0\quad {\rm as\ }M\to\fz,
\end{eqnarray*}
which, together with the completeness of $\hona$, implies that
\eqref{3.22x} holds true. This finishes the proof of Theorem \ref{tc.b}.
\end{proof}

\section{Equivalent Wavelet Characterizations of $\hona$\label{s4}}

\hskip\parindent In this section, we establish several equivalent
wavelet characterizations of $\hona$.
To this end, we first recall a version of the Khintchine inequality;
see, for example, \cite[Theorem 12.5.1]{g07}.

\begin{lemma}\label{ld.e}
Let $A$ be a countable index set and
$\boz$ be the product set $\{1,-1\}^{A}$, associated with
the Bernoulli probability measure $d\pp(\oz)$, namely, the product
$\prod_{a\in A}d\pp_{a}(\oz)$ of measures $d\pp_{a}(\oz)$
($a\in A$) such that $\pp_{a}(\{-1\})=1/2=\pp_{a}(\{1\})$,
where $\oz$ is an element of $\boz$ in the
form of $\{\oz(a)\}_{a\in A}\st\{-1,1\}$.
Suppose that $q\in(0,\fz)$. Then there exists a positive constant $C$ such that,
for all $\{\lz(a)\}_{a\in A}\st\cc$ and functions of the form,
$S(\oz):=\sum_{a\in A}\lz(a)w(a)$, it holds true that
$$
C^{-1}\lf[\sum_{a\in A}\lf|\lz(a)\r|^2\r]^{\frac12}
\le\lf[\int_{\boz}|S(\oz)|^q\,d\pp(\oz)\r]^{\frac1q}
\le C\lf[\sum_{a\in A}\lf|\lz(a)\r|^2\r]^{\frac12}.
$$
\end{lemma}

The following lemma is a slight variant of \cite[Corollary 7.10]{w97}.

\begin{lemma}\label{lc.x}
Suppose that $(\cx,d,\mu)$ is a metric measure space of homogeneous type,
$A$ is a countable index set and the series
$\sum_{a\in A}f_a$
converges unconditionally in $\lq$ with $q\in(0,\fz)$. Then
$$
\lf\|\lf(\sum_{a\in A}|f_a|^2\r)^{1/2}\r\|_{\lq}
\le\sup\lf\{\lf\|\sum_{a\in A}\ez_{a}f_a\r\|_{\lq}:\
\{\ez_a\}_{a\in A}\st\{-1,1\}\r\}<\fz,
$$
where the supremum is taken over all choices of
$\{\ez_a\}_{a\in A}\st\{-1,1\}$.
\end{lemma}

\begin{proof}
Let $q\in(0,\fz)$.
From the Khintchine inequality (Lemma \ref{ld.e}), the
Fubini-Tonelli theorem, the unconditional convergence of
$\sum_{a\in A}f_a$, \cite[Corollary 7.4]{w97} and
$\pp(\boz)=1$, it follows that
\begin{eqnarray*}
\lf\|\lf(\sum_{a\in A}\lf|f_a\r|^2\r)^{1/2}\r\|^q_{\lq}
&&\ls\int_{\cx}\int_{\boz}\lf|\sum_{a\in A}\oz(a)f_a(x)\r|^q\,d\pp(\oz)d\mu(x)\\
&&\sim\int_{\boz}\int_{\cx}\lf|\sum_{a\in A}\oz(a)f_a(x)\r|^q\,d\mu(x)d\pp(\oz)\\
&&\ls\sup\lf\{\lf\|\sum_{a\in A}\ez_{a}f_a\r\|^q_{\lq}:\
\{\ez_a\}_{a\in A}\st\{-1,1\}\r\}<\fz,
\end{eqnarray*}
which completes the proof of Lemma \ref{lc.x}.
\end{proof}

\begin{corollary}\label{cc.c}
Let $(\cx,d,\mu)$ be a metric measure space of homogeneous type.
Then there exists a positive constant $C$ such that, for all $f\in \hona$,
$$
\int_\cx\lf[\sum_{(k,\az,\bz)\in\sci}\lf|
\lf\langle f,\psi_{\az,\,\bz}^k\r\rangle\r|^2
\lf|\psi^k_{\az,\,\bz}(x)\r|^2\r]^{1/2}\,d\mu(x)
\le C\|f\|_{\hona},
$$
with $\sci$ as in Lemma \ref{lc.a}.
\end{corollary}

\begin{proof}
Let $f\in \hona$. From Theorem \ref{tc.b} and
$\hona\st L^1(\cx)$, we deduce that
$$
\sum_{(k,\,\az,\,\bz)\in\sci}
\lf\langle f,\psi^k_{\az,\,\bz}\r\rangle\psi^k_{\az,\,\bz}
$$
converges unconditionally in $\lon$.
For any sequence $\vec{\ez}:=\{\ez^k_{\az,\,\bz}\}_{(k,\az,\bz)\in\sci}
\st\{-1,1\}$, the operator $T_{\vec{\ez}}:\ \ltw\rightarrow\ltw$ is defined
by setting, for any $(k,\,\az,\,\bz)\in\sci$,
$$
T_{\vec{\ez}}\lf(\psi^k_{\az,\,\bz}\r):=\ez^k_{\az,\,\bz}\psi^k_{\az,\,\bz},
$$
which can be extended to an isometric isomorphism on $\ltw$.

Let $\{\sci_N\}_{N\in\nn}$ be any sequence of finite subsets of $\sci$
as in \eqref{3.24z}, $g\in\ltw$ and, for all $N\in\nn$,
$
g_N:=\sum_{(k,\,\az,\,\bz)\in\sci_N}
\lf\langle g,\psi^k_{\az,\,\bz}\r\rangle\psi^k_{\az,\,\bz}
$,
$$
K_{\vec{\ez},\,N}(x,y):=\sum_{(k,\,\az,\,\bz)\in\sci_N}
\ez^k_{\az,\,\bz}\psi^k_{\az,\,\bz}(x)\overline{\psi^k_{\az,\,\bz}(y)}
\quad {\rm for\ all\ }x,\,y\in\cx,
$$
and
$$
K_{\vec{\ez}}(x,y):=\sum_{(k,\,\az,\,\bz)\in\sci}
\ez^k_{\az,\,\bz}\psi^k_{\az,\,\bz}(x)\overline{\psi^k_{\az,\,\bz}(y)}
\quad {\rm for\ all\ }x,\,y\in\cx\ \mathrm{with}\ x\neq y.
$$

Now we claim that $K_{\vec{\ez}}$ is the Calder\'on-Zygmund kernel
of $T_{\vec{\ez}}$.
Indeed, by \cite[Proposition 10.3]{ah13}, we conclude that
$K_{\vec{\ez},\,N},\,K_{\vec{\ez}}
\in L^1_{\loc}(\{\cx\times\cx\}\bh\{(x,x):\ x\in\cx\})$
are Calder\'on-Zygmund kernels satisfying \eqref{d.a},
\eqref{d.b} and \eqref{d.g} with $s:=\eta$ and $c_{(K_{\vec{\ez},\,N})}$
and $C_{(K_{\vec{\ez},\,N})}$ independent of $N\in\nn$,
which, together with the boundedness of $T_{_{\vec{\ez}}}$ on $\ltw$,
the Lebesgue dominated convergence theorem and the Fubini theorem,
further implies that, for all $g,\,h\in C^{\eta}_b(\cx)$
with $\supp(g)\cap\supp(h)=\emptyset$,
\begin{eqnarray*}
\langle K_{\vec{\ez}},g\otimes h\rangle&&=\lim_{N\to\fz}
\lf\langle K_{\vec{\ez},\,N},g\otimes h\r\rangle\\
&&=\lim_{N\to\fz}\int_\cx\int_\cx
K_{\vec{\ez},\,N}(x,y)g(y)h(x)\,d\mu(y)d\mu(x)\\
&&=\lim_{N\to\fz}\lf(T_{\vec{\ez}}\lf(g_N\r),h\r)
=\lf(T_{\vec{\ez}}(g),h\r).
\end{eqnarray*}
Therefore, the above claim holds true, which, combined with
Theorem \ref{te.g}(i), further implies that, for all $f\in\hona$
and sequences $\vec{\ez}\st\{-1,1\}$,
$$
\lf\|T_{\vec{\ez}}(f)\r\|_{\lon}\ls\|f\|_{\hona}.
$$
From this and Lemma \ref{lc.x} with $q=1$, we further deduce that
\begin{eqnarray*}
&&\int_\cx\lf[\sum_{(k,\,\az,\,\bz)\in\sci}
\lf|\lf\langle f, \psi^k_{\az,\,\bz}\r\rangle\r|^2
\lf|\psi^k_{\az,\,\bz}(x)\r|^2\r]^{1/2}\,d\mu(x)\\
&&\noz\hs\ls\sup\lf\{\lf\|\sum_{(k,\,\az,\,\bz)\in\sci}\ez^k_{\az,\,\bz}
\lf\langle f, \psi^k_{\az,\,\bz}\r\rangle\psi^k_{\az,\,\bz}\r\|_{\lon}:\
\lf\{\ez^k_{\az,\,\bz}\r\}_{(k,\,\az,\,\bz)\in\sci}\st\{-1,1\}\r\}\\
&&\noz\hs\sim\sup\lf\{\lf\|T_{\vec{\ez}}(f)\r\|_{\lon}:\ \vec{\ez}\st\{-1,1\}\r\}
\ls\|f\|_{\hona},
\end{eqnarray*}
which completes the proof of Corollary \ref{cc.c}.
\end{proof}

Now we establish several equivalent characterizations
for $\hona$ in terms of wavelets.
To this end, we need more notation.
We point out that, for any $(k,\az,\bz)\in\sci$
with $\sci$ as in Lemma \ref{lc.a}, we have
 $\psi^k_{\az,\,\bz}\in\li$ and hence
$\langle f,\psi^k_{\az,\,\bz}\rangle$ is well defined for
any $f\in\lon$ in the sense of duality between $\lon$ and $\li$.

\begin{theorem}\label{tc.d}
Let $(\cx,d,\mu)$ be a metric measure space of homogeneous type.
Suppose that $f\in\lon$ and
$$
f=\sum_{(k,\,\az,\,\bz)\in\sci}
\langle f,\psi^k_{\az,\,\bz}\rangle\psi^k_{\az,\,\bz}
\quad {\rm in}\quad \lon.
$$
Then the following statements are mutually equivalent:
\begin{enumerate}
\item[{\rm (i)}] $f\in \hona$;

\item[{\rm (ii)}] $\sum_{(k,\,\az,\,\bz)\in\sci}
\langle f,\psi^k_{\az,\,\bz}\rangle\psi^k_{\az,\,\bz}$
converges unconditionally in $L^1(\cx)$;

\item[{\rm (iii)}] $\|f\|_{\rm (iii)}
:=\|\{\sum_{(k,\,\az,\,\bz)\in\sci}
|\langle f,\psi^k_{\az,\,\bz}\rangle|^2
|\psi^k_{\az,\,\bz}|^2\}^{1/2}\|_{L^1(\cx)}<\fz;
$

\item[{\rm (iv)}] $\|f\|_{\rm (iv)}:=
\|\{\sum_{(k,\,\az,\,\bz)\in\sci}
|\langle f,\psi^k_{\az,\,\bz}\rangle|^2
\frac{\chi_{Q_\az^k}}{\mu(Q_\az^k)}\}^{1/2}\|_{L^1(\cx)}<\fz;
$

\item[{\rm (v)}] $\|f\|_{\rm (v)}:=
\|\{\sum_{(k,\,\az,\,\bz)\in\sci}
|\langle f,\psi^k_{\az,\,\bz}\rangle|^2
[R^k_{\az,\,\bz}]^2\}^{1/2}\|_{L^1(\cx)}<\fz$,
\end{enumerate}
here and hereafter,
\begin{equation}\label{c.d}
R^k_{\az,\,\bz}:=\frac{\chi_{W^k_{\az,\,\bz}}}
{\sqrt{\mu(Q^k_\az)}},
\end{equation}
and
\begin{equation}\label{c.7}
W^k_{\az,\,\bz}:=B\lf(y^k_{\bz},\ez_0\dz^k\r)\st Q^k_\az
\end{equation}
as in Theorem \ref{tb.m}.

Moreover, $\|\cdot\|_{\rm (iii)}$, $\|\cdot\|_{\rm (iv)}$
and $\|\cdot\|_{\rm (v)}$ give norms on $\hona$, which are
equivalent to $\|\cdot\|_{\hona}$, respectively.
\end{theorem}

Before we prove Theorem \ref{tc.d},
we first establish several useful lemmas which are
of independent interest.

In what follows, let
\begin{equation}\label{c.8}
\cd:=\lf\{Q^k_{\az}\r\}_{(k,\,\az)\in\sca}
\end{equation}
be the dyadic system as in Theorem \ref{tb.c}.
The following notion of the dyadic maximal function is taken from \cite{abi1}.
Namely, for any $f\in L^1_{\loc}(\cx)$, the \emph{dyadic maximal function}
is defined by setting
$$
M^{dy}(f)(x):=\sup_{x\in Q\in\cd}\frac1{\mu(Q)}\int_Q
|f(y)|\,d\mu(y), \quad x\in\cx.
$$

The following lemma is on the boundedness of $M^{dy}(f)$,
whose proof is completely analogous to that of \cite[Theroem 3.1]{abi1},
the details being omitted.

\begin{lemma}\label{lc.g}
Let $(\cx,d,\mu)$ be a metric measure space of homogeneous type.
Then the following conclusions hold true:

{\rm (a)} For any $\lz\in(0,\fz)$ and $f\in\lon$, there exists a disjoint
family $\cf\st\cd$ such that
$$
\lf\{x\in\cx:\ M^{dy}(f)(x)>\lz\r\}=\bigcup_{Q\in\cf}Q;
$$

{\rm (b)} the weak type $(1,1)$ inequality
$$
\mu\lf(\lf\{x\in\cx:\ M^{dy}(f)(x)>\lz\r\}\r)
\le\frac{1}{\lz}\int_{\cx}|f(y)|\,d\mu(y)
$$
holds true for all $f\in L^1(\cx)$ and $\lz\in(0,\fz)$;

{\rm (c)} for any $p\in(1,\fz]$, there exists a positive constant
$C_{(p)}$, depending on $p$, such that, for all $f\in L^p(\cx)$,
$$
\lf\|M^{dy}(f)\r\|_{\lp}\le C_{(p)}\|f\|_{\lp}.
$$
\end{lemma}

\begin{remark}\label{rc.j}
Let $M$ be the \emph{Hardy-Littlewood maximal function} defined by setting
$$
M(f)(x):=\sup_{\{B\ni x:\ B\ {\rm ball}\}}\frac1{\mu(B)}\int_B|f(y)|\,d\mu(y)
$$
for all $f\in L^1_{\loc}(\cx)$ and $x\in\cx$.
It follows easily from Theorem \ref{tb.c}(iv) and \eqref{a.b}
that there exists a positive
constant $C$ such that, for all $f\in L^1_{\loc}(\cx)$,
$$M^{dy}(f)\le CM(f).$$
It is still unclear whether the inverse of the above
inequality holds true or not;
see \cite{abi1} for some comparisons between
the level sets of $M^{dy}$ and $M$.
\end{remark}

By Lemma \ref{lc.g}, the classical Lebesgue differentiation theorem associated
to the dyadic cubes on $\rr^D$ can be easily generalized to metric
measure spaces of homogeneous type as follows
(see, for example, the proof of \cite[Theorem 6.4]{w97} on $\rr^D$),
the details being omitted.

\begin{lemma}\label{lc.h}
Let $(\cx,d,\mu)$ be a metric measure space of homogeneous type and
$f\in\lon$. Then, for $\mu$-almost every $x\in\cx$ and
for every decreasing
sequence of dyadic cubes $\{Q_j\}_{j=1}^\fz\st\cd$ such that
$\bigcap_{j=1}^\fz Q_j=\{x\}$, it holds true that
$$
\lim_{j\to\fz}\frac{1}{\mu(Q_j)}\int_{Q_j}f(y)\,d\mu(y)=f(x).
$$
\end{lemma}

Now we introduce a key lemma, which is an extension of
\cite[Proposition 8.15]{w97} on $\rr^D$.

\begin{lemma}\label{lc.i}
Let $(\cx,d,\mu)$ be a metric measure space of homogeneous type.
For any family of numbers,
$\{a(j,\az,\bz)\}_{(j,\,\az,\,\bz)\in\sci}\st\cc$
with $\sci$ as in Lemma \ref{lc.a}, let $\cs$ be any finite
subset of $\sci$ and
$$
\vz_{\cs}(x):=\lf\{\sum_{(j,\,\az,\,\bz)\in\cs}
|a(j,\az,\bz)|^2
\lf[R^j_{\az,\,\bz}(x)\r]^2\r\}^{1/2},\quad x\in\cx,
$$
where $R^j_{\az,\,\bz}$ is as in \eqref{c.d}.
Suppose that and $\vz_{\cs}\in\lon$.
Then the function
$$
\sum_{(j,\,\az,\,\bz)\in\cs}a(j,\az,\bz)
\psi^j_{\az,\,\bz}\in \hona
$$
and there exists a positive constant $C$, independent of $\cs$,
such that
$$
 \lf\|\sum_{(j,\,\az,\,\bz)\in\cs}a(j,\az,\bz)
\psi^j_{\az,\,\bz}\r\|_{\hona}\le C\|\vz_{\cs}\|_{\lon}.
$$
\end{lemma}

\begin{proof}
In order to show this lemma, we write
$$
H:=\sum_{(j,\,\az,\,\bz)\in\cs}a(j,\az,\bz)
\psi^j_{\az,\,\bz}
$$
into a sum of molecules.  This will be done by partitioning the
index set $\cs$ into sets of $\{D(k,\tz)\}_{k\in\zz,\,\tz\in\scb_k}$,
where $\scb_k$ for any $k\in\zz$ denotes some
index set which will be determined later, in a way such that
$$
A_\tz^k:=\sum_{(j,\,\az,\,\bz)\in D(k,\,\tz)}a(j,\az,\bz)\psi^j_{\az,\,\bz}
$$
is an appropriate multiple of a $(1,2,\eta)$-molecule
centered at some ball $B$,
where $\eta$ and $B$ will also be determined later.

For any $k\in\zz$, let $\boz_k:=\{x\in\cx:\ \vz_{\cs}(x)>2^k\}$.
Obviously, $\boz_{k+1}\st\boz_k$ for all $k\in\zz$.
Thus, by this and the facts that
$\mu(\boz_{k+1})\le\|\vz_{\cs}\|_{\lon}/2^{(k+1)}\to 0$, as $k\to\fz$,
and $\bigcup_{k\in\zz}\boz_k=\cx$, we know that
\begin{eqnarray}\label{c.e}
\sum_{k=-\fz}^\fz2^{k}\mu(\boz_k)
&&=\sum_{k=-\fz}^\fz2^{k}\sum_{j=k}^\fz\mu(\boz_j\bh\boz_{j+1})\\
&&\noz\le\sum_{k=-\fz}^\fz2^{k}\sum_{j=k}^\fz2^{-j}
\int_{\boz_j\bh\boz_{j+1}}\vz_{\cs}(x)\,d\mu(x)\\
&&\noz=\sum_{j=-\fz}^\fz\sum_{k=-\fz}^j2^{(k-j)}
\int_{\boz_j\bh\boz_{j+1}}\vz_{\cs}(x)\,d\mu(x)\\
&&\noz\sim\sum_{j=-\fz}^\fz
\int_{\boz_j\bh\boz_{j+1}}\vz_{\cs}(x)\,d\mu(x)
\sim\int_{\cx}\vz_{\cs}(x)\,d\mu(x).
\end{eqnarray}

For any $k\in\zz$, let
$$
\ccc_k:=\lf\{(j,\az,\bz)\in\cs:\
\mu\lf(\boz_k\cap Q^j_\az\r)>\frac1{2C_2}\mu\lf(Q^j_\az\r)\r\},
$$
where $C_2\in[1,\fz)$ is a constant, independent of
$j$, $\az$ and $\bz$, satisfying
\begin{equation}\label{b.u}
\mu(Q_\az^j)\le C_2\mu(W^j_{\az,\,\bz})
\end{equation}
with $W^j_{\az,\,\bz}$ defined as in \eqref{c.7}
(see Remark \ref{rb.z}).
From the decreasing property of $k\mapsto\boz_k$, we deduce that
$\ccc_k\supset\ccc_{k+1}$ for all $k\in\zz$.
Define
\begin{equation}\label{3.31x}
\boz_k^*:=\bigcup_{(j,\,\az,\,\bz)\in\ccc_k}Q_\az^j.
\end{equation}
Now we choose a sequence of decreasing dyadic cubes,
$\{Q^j_{\az(j)}\}_{j\in\nn}\st\cd$,
where $\cd$ is as in \eqref{c.8} and
$\az(j)\in\sca_j$ with $\sca_j$ as in \eqref{b.v}
 for all $j\in\nn$ such that
$\bigcap_{j=1}^\fz Q^j_{\az(j)}=\{x\}$.
Indeed, by Theorem \ref{tb.c}(iii), we see that $x\in\cx=\bigcup_{\az\in\sca_1}Q^j_\az$.
Thus, there exists $\az(1)\in\sca_1$ such that $x\in Q^1_{\az(1)}$.
Moreover, from Remark \ref{rb.d}(ii), we deduce that
$x\in Q^1_{\az(1)}=\bigcup_{\az\in L(1,\az(1))}Q^2_\az$, which further
implies that there exists $\az(2)\in L(1,\az(1))$ such that
$x\in Q^2_{\az(2)}\st Q^1_{\az(1)}$. Repeating this procedure, we
obtain a decreasing sequence of dyadic cubes,
$\{Q^j_{\az(j)}\}_{j\in\nn}\st\cd$, satisfying
$\bigcap_{j=1}^\fz Q^j_{\az(j)}\supset\{x\}$.
On the other hand, by Theorem \ref{tb.c}(iv) with $C_1:=4$, we see that
$\bigcap_{j=1}^\fz Q^j_{\az(j)}\st\bigcap_{j=1}^\fz B(x^j_{\az(j)},4\dz^j)$.
Now we claim that $\bigcap_{j=1}^\fz B(x^j_{\az(j)},4\dz^j)=\{x\}$.
Obviously, by Theorem \ref{tb.c}(v), we have $B(x^{j+1}_{\az(j+1)},4\dz^{j+1})
\st B(x^j_{\az(j)},4\dz^j)$ for any $j\in\nn$, and
 $\bigcap_{j=1}^\fz B(x^j_{\az(j)},4\dz^j)\supset\{x\}$.
Conversely, if $y\in\bigcap_{j=1}^\fz B(x^j_{\az(j)},4\dz^j)$, then
$$
d(x,y)\le d\lf(x,x^j_{\az(j)}\r)+d\lf(x^j_{\az(j)},y\r)
<8\dz^j\to0
$$
as $j\to\fz$.
This shows that $y=x$ and hence the above claim, which further
implies that $\bigcap_{j=1}^\fz Q^j_{\az(j)}=\{x\}$.

By Lemma \ref{lc.h}, we know that, for $\mu$-almost every $x\in\cx$,
$$
\lim_{j\to\fz}\frac1{\mu(Q^j_{\az(j)})}\int_{Q^j_{\az(j)}}
\chi_{\boz_k}(y)\,d\mu(y)=\chi_{\boz_k}(x).
$$
Thus, for $\mu$-almost every $x\in\boz_k$, there exists $j_0\in\nn$ such that
$x\in Q^{j_0}_{\az(j_0)}$ and
\begin{equation}\label{c.f}
\mu\lf(\boz_k\cap Q^{j_0}_{\az(j_0)}\r)>\frac1{2C_2}
\mu\lf(Q^j_\az\r),
\end{equation}
which further implies that $Q^{j_0}_{\az(j_0)}\in\ccc_k$
and $x\in\boz_k^*$. That is, there exists a set $Z$ of measure zero such that
\begin{equation}\label{c.g}
\boz_k\st\boz_k^*\cup Z.
\end{equation}

For a fixed $k\in\zz$, let $\{Q(k,\tz)\}_{\tz\in\scb_k}
:=\{Q^{k(\tz)}_\tz\}_{\tz\in\scb_k}\st\cd$,
where $\cd$ is as in \eqref{c.8}
and $\scb_k$ denotes some unique index set such that
$\{Q(k,\tz)\}_{\tz\in\scb_k}$ is
the class of all maximal dyadic cubes in $\{Q^j_\az:\ (j,\az,\bz)\in\ccc_k\}$
and, for any $\tz\in\scb_k$, $k(\tz)$ denotes some integer depending on $\tz$.
It is easy to see that $\{Q(k,\tz)\}_{\tz\in\scb_k}\st\cd$ is pairwise
disjoint and
\begin{equation}\label{c.p}
\boz_k^*=\bigcup_{\tz\in\scb_k}Q(k,\tz).
\end{equation}
By this, \eqref{c.f} and \eqref{c.g}, we conclude that
\begin{eqnarray}\label{c.h}
\mu\lf(\boz_k^*\r)&&=\sum_{\tz\in\scb_k}\mu(Q(k,\tz))
\ls\sum_{\tz\in\scb_k}\mu(\boz_k\cap Q(k,\tz))\\
&&\noz\sim\mu\lf(\boz_k\cap\lf(\bigcup_{\tz\in\scb_k}Q(k,\tz)\r)\r)
\sim\mu\lf(\boz_k\cap\boz_k^*\r)\sim\mu(\boz_k).
\end{eqnarray}

Observe that, for any $(j,\,\az,\,\bz)\in\cs$,
if $a(j,\az,\bz)\neq0$, then there exists $\wz{k}\in\zz$, depending
on $j$, $\az$ and $\bz$, such that
$|a(j,\az,\bz)|>2^{\wz{k}}[\mu(Q^j_\az)]^{1/2}$.
Thus, for all $x\in W^j_{\az,\,\bz}$,
$\vz_{\cs}(x)>2^{\wz{k}}$, which shows that
$W^j_{\az,\,\bz}\st\boz_{\wz{k}}$. From this, \eqref{c.7}
and \eqref{b.u}, it follows that
$$
\mu\lf(Q^j_\az\cap\boz_{\wz{k}}\r)\ge\mu\lf(W^j_{\az,\,\bz}\r)\ge
\frac1{C_2}\mu\lf(Q^j_\az\r)>\frac1{2C_2}\mu\lf(Q^j_\az\r),
$$
which shows that $(j,\az)\in\ccc_{\wz{k}}$
and hence there exists $k\in\zz$ such that $(j,\az,\bz)\in\ccc_k\bh\ccc_{k+1}$.

Let $k\in\zz$, $\tz\in\scb_k$, $\ce_k:=\ccc_k\bh\ccc_{k+1}$ and
$$
D(k,\tz):=\lf\{(j,\,\az,\,\bz)\in\ce_k:\ Q^j_\az\st Q(k,\tz)\r\}.
$$
Now we claim that this is the desired splitting.
Indeed, for any $(j,\,\az,\,\bz)\in\cs$
such that $a(j,\az,\bz)\neq0$, by the above
proof, we know that there exists
$k\in\zz$ such that $(j,\az,\bz)\in\ccc_k\bh\ccc_{k+1}=:\ce_k$, which,
together with \eqref{3.31x} and \eqref{c.p}, further implies that
there exists $\tz\in\scb_k$ such that $Q^j_\az\st Q(k,\tz)$
and hence $(j,\az,\bz)\in D(k,\tz)$.
On the other hand, it is obvious that
$\bigcup_{k\in\zz,\,\tz\in\scb_k}D(k,\tz)\st\cs$.
Thus, to show the above claim, it suffices to prove that
$\{D(k,\tz)\}_{k\in\zz,\,\tz\in\scb_k}$ are mutually disjoint.
To this end, for $k,\,\wz{k}\in\zz$ and $\tz,\,\wz{\tz}\in\scb_k$,
if there exists $(j,\az,\bz)\in D(k,\tz)
\cap D(\wz{k},\wz{\tz})$, then, by the pairwise disjointness
of $\{\ce_k\}_{k\in\zz}$, we know that $k=\wz{k}$.
Moreover, from $Q^j_\az\st Q(k,\tz)\cap Q(k,\wz{\tz})\neq\emptyset$
and the maximality of $Q(k,\tz)$ and $Q(k,\wz{\tz})$, we
deduce that $Q(k,\tz)=Q(k,\wz{\tz})$ and hence
$\tz=\wz{\tz}$, which, combined with $k=\wz{k}$, further
implies that $D(k,\tz)=D(\wz{k},\wz{\tz})$. This finishes
the proof of the above claim.

As a consequence of the above claim, we have
\begin{equation}\label{c.i}
\sum_{(j,\,\az,\,\bz)\in\cs}
a(j,\az,\bz)\psi^j_{\az,\,\bz}
=\sum_{k\in\zz}\sum_{\tz\in\scb_k}A^k_\tz
\end{equation}
and, for all $k\in\zz$ and $\tz\in\scb_k$,
\begin{equation}\label{c.j}
\lf\|A^k_\tz\r\|^2_{\ltw}
=\sum_{(j,\,\az,\,\bz)\in D(k,\,\tz)}|a(j,\az,\bz)|^2.
\end{equation}
Let $k\in\zz$ and $\tz\in\scb_k$.
Observe that $(j,\az,\bz)\in D(k,\tz)$ implies that $Q^j_\az\st Q(k,\tz)$
and $Q^j_\az\not\in\ccc_{k+1}$. Thus,
$$
\mu\lf(Q^j_\az\bh\boz_{k+1}\r)=\mu\lf(Q_\az^j\r)-
\mu\lf(Q^j_\az\cap\boz_{k+1} \r)\ge\lf(1-\frac1{2C_2}\r)
\mu\lf(Q_\az^j\r).
$$
By the finiteness of $\cs$ and the above claim, we easily conclude that
there are only finitely many $D(k,\tz)\neq\emptyset$. Thus,
assuming that, if $D(k,\tz)=\emptyset$, then $A^k_\tz:=0$,
there are only finitely many $A^k_\tz$ in \eqref{c.i} are non-zero.

Since $[\vz_{\cs}(x)]^2\ge\sum_{(j,\,\az,\,\bz)\in D(k,\,\tz)}|a(j,\az,\bz)|^2
[R^j_{\az,\,\bz}(x)]^2$ for all $x\in\cx$,
where $R^j_{\az,\,\bz}$ is as in \eqref{c.d}, we have
\begin{eqnarray}\label{c.l}
&&\int_{Q(k,\,\tz)\bh\boz_{k+1}}[\vz_{\cs}(x)]^2\,d\mu(x)\\
&&\hs\noz\ge\sum_{(j,\,\az,\,\bz)\in D(k,\,\tz)}|a(j,\az,\bz)|^2
\int_{Q(k,\,\tz)\bh\boz_{k+1}}[R^j_{\az,\,\bz}(x)]^2\,d\mu(x)\\
&&\noz\hs=\sum_{(j,\,\az,\,\bz)\in D(k,\,\tz)}|a(j,\az,\bz)|^2
\frac{\mu(W^j_{\az,\,\bz}\cap[Q(k,\tz)\bh\boz_{k+1}])}{\mu(Q^j_\az)}.
\end{eqnarray}
By $W^j_{\az,\,\bz}\st Q^j_\az\st Q(k,\tz)$
[see \eqref{c.7}], we find that
$W^j_{\az,\,\bz}\cap [Q(k,\tz)\bh\boz_{k+1}]
=W^j_{\az,\,\bz}\bh\boz_{k+1}$.
From this, $Q^j_\az\not\in\ccc_{k+1}$ and \eqref{b.u}, it follows that
\begin{eqnarray}\label{c.m}
&&\mu\lf(W^j_{\az,\,\bz}\cap\lf[ Q(k,\tz)\bh\boz_{k+1}\r]\r)\\
&&\noz\hs=\mu\lf(W^j_{\az,\,\bz}\bh\boz_{k+1}\r)
=\mu\lf(W^j_{\az,\,\bz}\r)-\mu\lf(W^j_{\az,\,\bz}\cap\boz_{k+1}\r)\\
&&\noz\hs\ge\mu\lf(W^j_{\az,\,\bz}\r)
-\mu\lf(Q^j_\az\cap\boz_{k+1}\r)\ge\frac1{2C_2}\mu\lf(Q_\az^j\r).
\end{eqnarray}
Moreover, combining \eqref{c.j}, \eqref{c.l} and \eqref{c.m}, we conclude that
\begin{eqnarray}\label{c.n}
\lf\|A^k_\tz\r\|^2_{\ltw}&&\ls\int_{Q(k,\,\tz)\bh \boz_{k+1}}
[\vz_{\cs}(x)]^2\,d\mu(x)\\
&&\noz\ls2^{2(k+1)}\mu\lf(Q(k,\tz)\bh\boz_{k+1}\r)
\ls4^k\mu(Q(k,\tz)).
\end{eqnarray}
Thus, by \eqref{c.n} and $Q(k,\tz)\st B(x^{k(\tz)}_\tz,4\dz^{k(\tz)})$
[see Theorem \ref{tb.c}(iv)], we obtain
$$
\lf\|A^k_\tz\r\|^2_{\ltw}\ls\mu(Q(k,\tz))
\ls V\lf(x^{k(\tz)}_\tz,8\dz^{k(\tz)}\r).
$$

Let $\wz{A}^k_\tz:=A^k_\tz/\lz(k,\tz)$, where
$\lz(k,\tz):=[V(x^{k(\tz)}_\tz,8\dz^{k(\tz)})]^{1/p-1/2}
\|A^k_\tz\|_{\ltw}\in(0,\fz)$.
Now we claim that $\wz{A}^k_\tz$ is a $(p,2,\eta)$-molecule
centered at $B(x^{k(\tz)}_\tz,8\dz^{k(\tz)})$
multiplied by a positive constant, where
$\eta:=\{\eta_\ell\}_{\ell=1}^\fz$ and
$\eta_\ell:=2^{-\frac{\ell}{2}(M_0-1)}2^{\frac{n\ell}2}$
for any $\ell\in\nn$ and a fixed large enough constant $M_0$
satisfying $M_0>1+n+G_0$ with $n$ and $G_0$, respectively,
as in \eqref{a.b} and Remark \ref{rb.l}(ii),
$\sum_{\ell=1}^\fz\ell\eta_{\ell}<\fz$. Indeed, obviously, we have
\begin{equation}\label{3.39x2}
\lf\|\wz{A}^k_\tz\r\|_{\ltw}
=\lf[V(x^{k(\tz)}_\tz,8\dz^{k(\tz)})\r]^{1/2-1/p}.
\end{equation}
For any $\ell\in\nn$, by the Minkowski inequality and
the H\"older inequality, we see that
\begin{eqnarray*}
{\rm J}:&&=\lf\|\wz{A}^k_\tz\chi_{B(x^{k(\tz)}_\tz,2^\ell8\dz^{k(\tz)})\bh
B(x^{k(\tz)}_\tz,2^{\ell-1}8\dz^{k(\tz)})}\r\|_{\ltw}\\
&&\le\frac1{\lz(k,\tz)}\sum_{(j,\,\az,\,\bz)\in D(k,\,\tz)}
|a(j,\az,\bz)|\lf\|\psi^j_{\az,\,\bz}
\chi_{B(x^{k(\tz)}_\tz,2^\ell8\dz^{k(\tz)})\bh B(x^{k(\tz)}_\tz,
2^{\ell-1}8\dz^{k(\tz)})}\r\|_{\ltw}\\
&&\ls\frac1{\lz(k,\tz)}\lf[\sum_{(j,\,\az,\,\bz)\in D(k,\,\tz)}
|a(j,\az,\bz)|^2\r]^{1/2}\\
&&\hs\times\lf[\sum_{(j,\,\az,\,\bz)\in D(k,\,\tz)}
\lf\|\psi^j_{\az,\,\bz}
\chi_{B(x^{k(\tz)}_\tz,2^\ell8\dz^{k(\tz)})
\bh B(x^{k(\tz)}_\tz,2^{\ell-1}8\dz^{k(\tz)})}\r\|^2_{\ltw}\r]^{1/2}.
\end{eqnarray*}

Moreover, for any $(j,\az,\bz)\in D(k,\tz)$, by Theorem \ref{tb.c}(iv)
and \eqref{2.10x}, we have
$$
x^{j+1}_{\bz}\in Q^{j+1}_\bz\st Q^j_\az\st Q(k,\tz)
\st B\lf(x^{k(\tz)}_\tz,4\dz^{k(\tz)}\r)
$$
and hence $d(x^{j+1}_{\bz},x^{k(\tz)}_\tz)<4\dz^{k(\tz)}$.
From this, we deduce that, for any $(j,\az,\bz)\in D(k,\tz)$ and
$
x\in B(x^{k(\tz)}_\tz,2^\ell8\dz^{k(\tz)})\bh
B(x^{k(\tz)}_\tz,2^{\ell-1}8\dz^{k(\tz)}),
$
$$
d\lf(x^{j+1}_{\bz},x\r)\ge d\lf(x,x^{k(\tz)}_\tz\r)
-d\lf(x^{k(\tz)}_\tz,x^{j+1}_{\bz}\r)>
2^{\ell+2}\dz^{k(\tz)}-4\dz^{k(\tz)}\ge2^{\ell+1}\dz^{k(\tz)},
$$
which, together with \eqref{b.a}, \eqref{a.b} and $k(\tz)\ge j$,
further implies that
\begin{eqnarray*}
&&\lf\|\psi^j_{\az,\,\bz}\chi_{B(x^{k(\tz)}_\tz,2^\ell8\dz^{k(\tz)})\bh
B(x^{k(\tz)}_\tz,2^{\ell-1}8\dz^{k(\tz)})}\r\|^2_{\ltw}\\
&&\hs\ls\int_{B(x^{k(\tz)}_\tz,2^\ell8\dz^{k(\tz)})
\bh B(x^{k(\tz)}_\tz,2^{\ell-1}8\dz^{k(\tz)})}
\frac{1}{V(x_{\bz}^{j+1},\dz^j)}
e^{-2\nu\dz^{-j}d(x_{\bz}^{j+1},\,x)}\,d\mu(x)\\
&&\hs\ls e^{-2^{\ell+2}\nu\dz^{k(\tz)-j}}
\frac{V(x_{\tz}^{k},2^\ell8\dz^{k(\tz)})}{V(x_{\bz}^{j+1},\dz^j)}
\ls e^{-2^{\ell+2}\nu\dz^{k(\tz)-j}}2^{n\ell}
\frac{V(x_{\tz}^{k},8\dz^{k(\tz)})}{V(x_{\bz}^{j+1},\dz^j)}\\
&&\hs\ls e^{-2^{\ell}\nu\dz^{k(\tz)-j}}2^{n\ell}
\frac{V(x_{\bz}^{j+1},12\dz^{k(\tz)})}{V(x_{\bz}^{j+1},\dz^j)}
\ls e^{-2^{\ell}\nu\dz^{k(\tz)-j}}2^{n\ell}\dz^{[k(\tz)-j]n}
\frac{V(x_{\bz}^{j+1},12\dz^j)}{V(x_{\bz}^{j+1},\dz^j)}\\
&&\hs\ls e^{-2^{\ell}\nu\dz^{k(\tz)-j}}2^{n\ell}\dz^{[k(\tz)-j]n}.
\end{eqnarray*}
By this, $D(k,\tz)\st\{(j,\,\az,\,\bz)\in\cs:\ j\ge k(\tz),\
d(x^j_\az,x^{k(\tz)}_\tz)<4\dz^{k(\tz)}\}$,
(i) and (iii) of Remark \ref{rb.l} and $M_0>G_0+n+1$,
we conclude that
\begin{eqnarray*}
&&\sum_{(j,\,\az,\,\bz)\in D(k,\,\tz)}
\lf\|\psi^j_{\az,\,\bz}
\chi_{B(x^{k(\tz)}_\tz,2^\ell\dz^{k(\tz)})
\bh B(x^{k(\tz)}_\tz,2^{\ell-1}\dz^{k(\tz)})}\r\|^2_{\ltw}\\
&&\quad\ls\sum_{j=k(\tz)}^\fz\sum_{\{\az\in\sca_j:
\ d(x^j_\az,\,x^{k(\tz)}_\tz)<4\dz^{k(\tz)}\}}
e^{-2^{\ell}\nu\dz^{k(\tz)-j}}2^{n\ell}\dz^{[k(\tz)-j]n}\\
&&\quad\ls\sum_{j=k(\tz)}^\fz
2^{-M_0\ell}\dz^{-M_0[k(\tz)-j]}\dz^{G_0[k(\tz)-j]}2^{n\ell}\dz^{[k(\tz)-j]n}\\
&&\quad\ls2^{-M_0\ell}2^{n\ell}\sum_{j=k(\tz)}^\fz\dz^{(G_0+n-M_0)[k(\tz)-j]}
\ls2^{-\ell}\eta^2_{\ell},
\end{eqnarray*}
which further implies that
\begin{equation}\label{3.39x3}
{\rm J}\ls\eta_\ell2^{-\ell/2}\frac{\|A^k_\tz\|_{\ltw}}{\lz(k,\tz)}
\sim\eta_\ell2^{-\ell/2}\lf[V\lf(x^{k(\tz)}_\tz,8\dz^{k(\tz)}\r)\r]^{1/2-1/p}.
\end{equation}

By \eqref{b.c} and the finiteness of $\cs$, we obtain
$$
\int_\cx\wz{A}^k_\tz(x)\,d\mu(x)=0.
$$
From this, \eqref{3.39x2} and \eqref{3.39x3},
together with $\sum_{\ell=1}^\fz\ell\eta_{\ell}<\fz$,
we deduce that the above claim holds true.

By the above claim, \eqref{c.i},
\begin{equation}\label{3.42x}
\sum_{(j,\,\az,\,\bz)\in\cs}
a(j,\az,\bz)\psi^j_{\az,\,\bz}
=\sum_{k\in\zz}\sum_{\tz\in\scb_k}\lz(k,\tz)\wz{A}^k_\tz
\end{equation}
with only finitely many $\lz(k,\tz)\wz{A}^k_\tz\neq0$,
and Theorem \ref{tc.y}, we conclude that
$$
\sum_{(j,\,\az,\,\bz)\in\cs}
a(j,\az,\bz)\psi^j_{\az,\,\bz}\in\hona.
$$

Moreover, by \eqref{c.n}, $Q(k,\tz)=Q^{k(\tz)}_\tz$, together with
Theorem \ref{tb.c}(iv), \eqref{a.b}, disjoint property of
$\{Q(k,\tz)\}_{\tz\in\scb_k}$, \eqref{c.h} and \eqref{c.e}, we conclude that
\begin{eqnarray*}
\sum_{k\in\zz}\sum_{\tz\in\scb_k}\lz(k,\tz)
&&=\sum_{k\in\zz}\sum_{\tz\in\scb_k}
\lf[V\lf(x^{k(\tz)}_\tz,8\dz^{k(\tz)}\r)\r]^{1/2}
\lf\|A^k_\tz\r\|_{\ltw}\\
&&\ls\sum_{k\in\zz}\sum_{\tz\in\scb_k}
\lf[V\lf(x^{k(\tz)}_\tz,8\dz^{k(\tz)}\r)\r]^{1/2}2^k\sqrt{\mu(Q(k,\tz))}\\
&&\ls\sum_{k\in\zz}2^k\sum_{\tz\in\scb_k}\mu(Q(k,\tz))
\ls\sum_{k\in\zz}2^k\mu\lf(\boz_k^*\r)\\
&&\ls\sum_{k\in\zz}2^{k}\mu\lf(\boz_k\r)
\sim\int_\cx\vz_{\cs}(x)\,d\mu(x)<\fz.
\end{eqnarray*}

Thus, from this, \eqref{3.42x} and Theorem \ref{tc.y}, it follows that
$$
\lf\|\sum_{(j,\,\az,\,\bz)\in\cs}
a(j,\az,\bz)\psi^j_{\az,\,\bz}\r\|_{\hona}
\ls\sum_{k\in\zz}\sum_{\tz\in\scb_k}\lz(k,\tz)
\ls\lf\|\vz_{\cs}\r\|_{\lon},
$$
which completes the proof of Lemma \ref{lc.i}.
\end{proof}

Now we are ready to prove Theorem \ref{tc.d}.

\begin{proof}[Proof of Theorem \ref{tc.d}]
Let $f\in\lon$ and
\begin{equation}\label{3.43x}
f=\sum_{(k,\,\az,\,\bz)\in\sci}
\langle f,\psi^k_{\az,\,\bz}\rangle\psi^k_{\az,\,\bz}
\quad {\rm in}\quad \lon.
\end{equation}
From $\hona\st\lon$ and Theorem \ref{tc.b},
we deduce that (i) implies (ii).

By Lemma \ref{lc.x}, we know that
(ii) implies (iii).

Now we prove ``$\rm (iii)\Longrightarrow(i)$".
Indeed, let $\{\sci_N\}_{N\in\nn}$ be any sequence
of finite subsets of $\sci$ as in \eqref{3.24z} and
$$
S_N(f):=\sum_{(k,\,\az,\,\bz)\in\sci_N}
\langle f,\psi^k_{\az,\,\bz}\rangle\psi^k_{\az,\,\bz},
\quad N\in\nn.
$$
For any $N,\,M\in\nn$ with $M<N$, by Theorem \ref{tb.m}, we have
\begin{eqnarray}\label{c.q}
&&\lf\{\sum_{(k,\,\az,\,\bz)\in\sci_N\bh\sci_M}
\lf|\lf\langle f,\psi^k_{\az,\,\bz}\r\rangle\r|^2
\lf[R^k_{\az,\,\bz}\r]^2\r\}^{1/2}\\
&&\noz\hs\ls\lf\{\sum_{(k,\,\az,\,\bz)\in\sci_N\bh\sci_M}
\lf|\lf\langle f,\psi^k_{\az,\,\bz}\r\rangle\r|^2
\lf|\psi^k_{\az,\,\bz}\r|^2\r\}^{1/2}\\
&&\noz\hs\ls\lf\{\sum_{(k,\,\az,\,\bz)\in\sci}
\lf|\lf\langle f,\psi^k_{\az,\,\bz}\r\rangle\r|^2
\lf|\psi^k_{\az,\,\bz}\r|^2\r\}^{1/2}\in\lon,
\end{eqnarray}
which, together with Lemma \ref{lc.i}, further implies that
$$
\lf\|S_N(f)-S_M(f)\r\|_{\hona}
\ls\lf\|\lf\{\sum_{(k,\,\az,\,\bz)\in\sci_N\bh\sci_M}
\lf|\lf\langle f,\psi^k_{\az,\,\bz}\r\rangle\r|^2
\lf[R^k_{\az,\,\bz}\r]^2\r\}^{1/2}\r\|_{\lon}\to0,
$$
as $N,\,M\to\fz$.

Thus, $\{S_N(f)\}_{N\in\nn}$ is a Cauchy sequence in $\hona$ and hence,
by Remark \ref{rc.i}, there exists $g\in\hona$ such that
$$
g=\lim_{N\to\fz} S_N(f)\quad {\rm in}\quad \hona.
$$
From this, the fact that $\hona\st\lon$
and \eqref{3.43x}, we deduce that
$$
g=\lim_{N\to\fz} S_N(f)=f\quad {\rm in}\quad \lon,
$$
which, combined with $g\in\hona$, further implies that $f\in\hona$.
This finishes the proof of ``$\rm (iii)\Longrightarrow(i)$"
and hence (i), (ii) and (iii) are mutually equivalent.

``$\rm (iii)\Longrightarrow(v)$" follows from Theorem \ref{tb.m}.

``$\rm (v)\Longrightarrow(i)$" is an implicit consequence of the proof
of ``$\rm (iii)\Longrightarrow(i)$". Thus, (i), (ii), (iii) and (v)
are mutually equivalent.

``$\rm (iv)\Longrightarrow(v)$" is obvious by \eqref{c.7}.

To show ``$\rm (v)\Longrightarrow(iv)$", we first claim that, for all
$s\in(0,\fz)$ and $(k,\az,\bz)\in\sci$,
\begin{equation}\label{3.44x}
\chi_{Q_\az^k}\ls\lf[M\lf(\chi_{W^k_{\az,\bz}}\r)^s\r]^{1/s}.
\end{equation}
Indeed, by Remark \ref{rb.z}, Theorem \ref{tb.c}(iv) and \eqref{a.b},
we know that, for all $x\in Q_\az^k\st B(x^k_\az,4\dz^k)$,
\begin{eqnarray*}
1&&\sim\lf[\frac{\mu(W^k_{\az,\,\bz})}{\mu(Q^k_\az)}\r]^{1/s}
\sim\lf\{\frac1{\mu(Q^k_\az)}\int_{Q^k_\az}
\lf[\chi_{W^k_{\az,\,\bz}}(y)\r]^s\,d\mu(y)\r\}^{1/s}\\
&&\ls\lf\{\frac1{\mu(B(x^k_\az,(1/3)\dz^k))}\int_{B(x^k_\az,\,4\dz^k)}
\lf[\chi_{W^k_{\az,\,\bz}}(y)\r]^s\,d\mu(y)\r\}^{1/s}
\ls\lf[M\lf(\chi_{W^k_{\az,\bz}}\r)^s(x)\r]^{1/s},
\end{eqnarray*}
which shows the above claim.

Moreover, by \eqref{3.44x}, with $s:=2/r$ and $r\in(0,1)$, and
the Fefferman-Stein vector-valued maximal function inequality
(see, for example, \cite[Theorem 1.2]{gly}), we obtain
\begin{eqnarray*}
&&\lf\|\lf\{\sum_{(k,\,\az,\,\bz)\in\sci}
\lf|\lf\langle f,\psi^k_{\az,\,\bz}\r\rangle\r|^2\lf[\mu\lf(Q_\az^k\r)\r]^{-1}
\chi_{Q_\az^k}\r\}^{1/2}\r\|_{\lon}\\
&&\hs\ls\int_\cx\lf\{\sum_{(k,\,\az,\,\bz)\in\sci}
\frac{|\langle f,\psi^k_{\az,\,\bz}\rangle|^2}{\mu(Q_\az^k)}
\lf[M\lf(\chi_{W^k_{\az,\,\bz}}\r)^{r/2}(x)\r]^{2/r}\r\}^{1/2}\,d\mu(x)\\
&&\hs\sim\lf\|\lf\{\sum_{(k,\,\az,\,\bz)\in\sci}
\lf[M\lf(\frac{|\langle f,\psi^k_{\az,\,\bz}\rangle|}
{[\mu(Q_\az^k)]^{1/2}}
\chi_{W^k_{\az,\,\bz}}\r)^{r}\r]^{2/r}\r\}^{r/2}\r\|_{L^{1/r}(\cx)}^{1/r}\\
&&\hs\ls\lf\|\lf\{\sum_{(k,\,\az,\,\bz)\in\sci}
\lf|\lf\langle f,\psi^k_{\az,\,\bz}\r\rangle\r|^2
\lf[R^k_{\az,\,\bz}\r]^2\r\}^{1/2}\r\|_{L^1(\cx)},
\end{eqnarray*}
which shows ``$\rm (v)\Longrightarrow(iv)$".
Thus, (i) through (v) are mutually equivalent.

Finally, we show that $\|\cdot\|_{\rm (iii)}$, $\|\cdot\|_{\rm (iv)}$
and $\|\cdot\|_{\rm (v)}$ give norms on $\hona$, which are
equivalent to $\|\cdot\|_{\hona}$, respectively.
Indeed, by \eqref{c.q}, Corollary \ref{cc.c},
Theorem \ref{tc.b} and Lemma \ref{lc.i}
we conclude that, for all $f\in\hona$,
\begin{eqnarray*}
\|f\|_{\rm (v)}&&\ls\|f\|_{\rm (iii)}\sim
\lf\|\lf\{\sum_{(k,\,\az,\,\bz)\in\sci}
\lf|\lf\langle f,\psi^k_{\az,\,\bz}\r\rangle\r|^2
\lf|\psi^k_{\az,\,\bz}(x)\r|^2\r\}^{1/2}\r\|_{\lon}
\\
&&\ls\|f\|_{\hona}
\sim\lim_{N\to\fz}\lf\|\sum_{(k,\,\az,\,\bz)\in\sci_N}
\lf\langle f,\psi^k_{\az,\,\bz}\r\rangle
\psi^k_{\az,\,\bz}(x)\r\|_{\hona}\\
&&\ls\lim_{N\to\fz}\|\vz_{\sci_N}\|_{\lon}\ls\|f\|_{\rm (v)}.
\end{eqnarray*}
Thus, $\|\cdot\|_{\rm (v)}\sim\|\cdot\|_{\hona}\sim\|\cdot\|_{\rm (iii)}$.
Moreover, by the proof of ``$\rm (iv)\Longleftrightarrow(v)$",
we see that
$\|\cdot\|_{\rm (iv)}\sim\|\cdot\|_{\rm (v)}$,
which implies the desired conclusion and hence
completes the proof of Theorem \ref{tc.d}.
\end{proof}

\begin{remark}\label{rc.x}
By arguments essentially the same
as those used in the case of $d$,
we conclude that all the results obtained in this article
\emph{remain} valid with the metric $d$ replaced by a
quasi-metric $\rho$, since most of the tools we need are from
\cite{ah13,ah15},
which were established in the context of
spaces of homogeneous type.
Some minor modifications are needed when dealing with
the inclusion relations between two balls, where
the quasi-triangle constant is involved,
which only alter the corresponding results
by additive positive constants via \eqref{a.b}.
\end{remark}

\bigskip

Xing Fu and Dachun Yang (Corresponding author)

\medskip

School of Mathematical Sciences, Beijing Normal University,
Laboratory of Mathematics and Complex Systems, Ministry of
Education, Beijing 100875, People's Republic of China

\smallskip

{\it E-mails}: \texttt{xingfu@mail.bnu.edu.cn} (X. Fu)

\hspace{1.55cm}\texttt{dcyang@bnu.edu.cn} (D. Yang)

\end{document}